\documentclass{amsart}
\usepackage{bbm}
\usepackage{mathabx}
\usepackage{tensor}
\usepackage{enumerate}
\usepackage{amssymb}
\usepackage{mathtools}
\usepackage{amsmath}
\usepackage{amsfonts}
\usepackage{amsthm}
\usepackage{stmaryrd}
\usepackage{paralist}
\usepackage{indentfirst}
\usepackage[all]{xy}
\usepackage{mathrsfs}
\usepackage{graphicx}
\usepackage{hyperref}
\usepackage{color}
\usepackage{tikz-cd}
\usepackage{tkz-euclide}
\usepackage{tikz}
\usetikzlibrary{matrix}
\usepackage{enumitem}
\usepackage{amsmath,amscd}
\usepackage{pifont}
\usetikzlibrary{arrows}
\usepackage[all]{xy}
\usepackage{graphicx}
\usepackage{placeins}
\usepackage{xspace}

\usetikzlibrary{calc,intersections,through,backgrounds}

\newcolumntype{S}{>{\centering\arraybackslash} m{.5\linewidth} }
\newcolumntype{s}{>{\centering\arraybackslash} m{.2\linewidth} }

\newtheorem{thm}{Theorem}[section]
\newtheorem{theorem}[thm]{Theorem}
\newtheorem{conjecture}[thm]{Conjecture}

\newtheorem{proposition}[thm]{Proposition}

\newtheorem{lemma}[thm]{Lemma}

\newtheorem{innercustomgeneric}{\customgenericname}
\providecommand{\customgenericname}{}
\newcommand{\newcustomtheorem}[2]{%
	\newenvironment{#1}[1]
	{%
		\renewcommand\customgenericname{#2}%
		\renewcommand\theinnercustomgeneric{##1}%
		\innercustomgeneric
	}
	{\endinnercustomgeneric}
}

\newcustomtheorem{customcri}{Criterion}

\theoremstyle{definition}

\newtheorem{definition}[thm]{Definition}

\newtheorem{notation}[thm]{Notation}

\numberwithin{equation}{section}

\numberwithin{mytheorem}{subsection}

\numberwithin{mytheorem}{subsection}
\numberwithin{myconjecture}{subsection}
\numberwithin{mydefinition}{subsection}
\numberwithin{myremark}{subsection}
\numberwithin{mysituation}{subsection}
\numberwithin{myhypothesis}{subsection}
\numberwithin{myquestion}{subsection}
\numberwithin{mynotation}{subsection}
\numberwithin{myfact}{subsection}
\numberwithin{myexamples}{subsection}
\numberwithin{myexample}{subsection}
\numberwithin{myconstruction}{subsection}
\numberwithin{mycaution}{subsection}
\numberwithin{myproposition}{subsection}
\numberwithin{mylemma}{subsection}
\numberwithin{mycorollary}{subsection}

\def\CC{\mathbb{C}}

\def\FF{\mathbb{F}}

\def\NN{\mathbb{N}}

\def\ZZ{\mathbb{Z}}

\def\calK{\mathcal{K}}

\def\calO{\mathcal{O}}

\newcommand{\tA}{\widetilde{A}}

\newcommand{\tN}{\widetilde{N}}

\newcommand{\teta}{\widetilde{\teta}}

\def\v{\mathrm{val}_\mu}
\def\vp{\mathrm{val}_p}

\def\Fp{\FF_p}

\newcommand{\mK}{{\mathfrak{m}_{\cK}}}

\renewcommand{\=}{:=}

\DeclareMathOperator{\val}{val}

\newcommand{\ua}{\underline{a}}

\newcommand{\ualpha}{\underline{\alpha}}
\newcommand{\ubeta}{\underline{\beta}}
\newcommand{\ugamma}{\underline{\gamma}}
\newcommand{\uxi}{\underline{\xi}}
\newcommand{\valmu}{\val_{\mu}}
\newcommand{\whc}{\widehat{c}}
\newcommand{\whf}{\widehat{f}}

\newcommand{\cK}{\mathcal{K}}

\newcommand{\OK}{O_{\cK}}
\newcommand{\Kres}{\widetilde{\cK}}

\pgfkeys{tikz/mymatrixenv/.style={decoration=brace,every left delimiter/.style={xshift=3pt},every right delimiter/.style={xshift=-3pt}}}
\pgfkeys{tikz/mymatrix/.style={matrix of math nodes,left delimiter=[,right delimiter={]},inner sep=2pt,column sep=1em,row sep=0.5em,nodes={inner sep=0pt}}}
\pgfkeys{tikz/mymatrixbrace/.style={decorate,thick}}

\begin{document}\large
	\title[Non-Linearizability for power series over fields of char $p$]{Non-Linearizability of power series over complete non-Archimedean fields of  
	positive	characteristic}
	
	\author{Rufei Ren}
	\address{Department of Mathematics, Fudan University\\220 Handan Rd., Yangpu District, Shanghai 200433, China.}
	\email{rufeir@fudan.edu.cn}
	\date{\today}
	
	\keywords{Non-Linearizability}
	\maketitle
	\begin{abstract}
		In \cite{HY83}, Herman and Yoccoz prove that for any 
		given locally analytic (at $z=0$) power series $f(z)=z(\lambda +\sum_{i=1}^\infty a_iz^i)$ over a complete non-Archimedean field of characteristic $0$ if $|\lambda|=1$ and $\lambda$ is not a root of unity, then $f$ is locally linearizable at $z=0$. They ask the same question for power series over fields of positive characteristic. 
		
		In this paper, we prove that, on opposite, most such power series in this case are more likely to be non-linearizable. 
		More precisely, given a complete non-Archimedean field $\calK$ of positive characteristic and a power series
		$f(z)=z(\lambda +\sum_{i=1}^\infty a_iz^i) \in \calK[\![z]\!]$ with $\lambda$ not a root of unity and $|1-\lambda|<1$,
		we prove a sufficient condition (Criterion~\ref{c:key criterion}) for $f$ to be non-linearizable. 
		This phenomenon of prevalence for power series over fields of positive characteristic being non-linearizable  was initially conjectured in \cite[p 147]{Her87} by Herman, and formulated into a concrete question by Lindahl as \cite[Conjecture~2.2]{Lin04}.  
		
		As applications of Criterion~\ref{c:key criterion}, we prove the non-linearizability of three families of polynomials.
	\end{abstract}
	\setcounter{tocdepth}{1}
	\tableofcontents
	
	\section{Introduction}
	
We first introduce the definition of linearizability of a power series around the origin.
Given any  complete valued field $K$, and  a locally analytic (at $z=0$) power series $f(z)=z(\lambda +\sum_{i=1}^\infty a_iz^i )\in K[\![z]\!]$, we call $f$ \emph{locally linearizable at $z=0$} (or \emph{linearizable} for short) if there 
	exists a locally analytic (at $z=0$) power series $h(z)=b_0z+b_1z^2+\cdots\in K[\![z]\!]$ such that $h\circ f \circ h^{-1}(z)=\lambda z$; otherwise, we call it \emph{non-linearizable}.
	
		The question on linearizability of a power series at its periodic points is first introduced by Poincar\'e from studying the stability of systems of differential equations. Gradually, people began to realize its importance and started to research at it. Until now, it is still a popular topic that a lot mathematicians are working on. In the very beginning of the history of the study of  linearizability,  
	people focused on the power series over the complex field, which makes sense since it  relates to the reality most closely. Among these early important results, most of them give sufficient conditions for a power series to be linearizable. For example, 
	in 1942, Siegel proved in his
	famous paper \cite{SI} that the condition
	\begin{equation}\label{v1}
		|1-\lambda^n|\geq Cn^{-\beta} \textrm{\ for some real numbers\ } C, \beta > 0
	\end{equation} 
	on $\lambda$ implies the linearizability of $f$, which is so-called  the ``Siegel's condition''. Note that even though the Siegel's condition was initially used to prove linearizability for power series over $\CC$, it in fact works  on any complete valued fields. 
	Later on, Brjuno in \cite{Brj} proved that the weaker condition 
	\begin{equation}\label{v2}
		-\sum_{k=0}^{\infty} 2^{-k} \log \left(\inf _{1 \leqslant n \leqslant 2^{k+1}-1}\left|1-\lambda^{n}\right|\right)<+\infty
	\end{equation} 
	is enough to imply the linearizability of $f$. Unfortunately, \eqref{v1} and \eqref{v2} do not exhaust all $\lambda$ and a complete description is still open. The only case over $\CC$ that are fully settled is the family of quadratic polynomials. It is proved by Yoccoz in \cite{Yoc95b}  
	that a quadratic polynomial $\lambda z+a_1z^2\in \CC[z]$ is linearizable if and only if $\lambda$ satisfies \eqref{v2}.
	
It worth mentioning that if $\lambda$ is a root of unity, then $f$ is not linearizable. This was first proved by Herman and Yoccoz in \cite{HY83}, and by Rivera-Letelier in \cite{RL03} with a different approach for dealing more general cases. We also note a direct consequence from Siegel's condition that if $|\lambda|\neq 1$, then $f$ is linearizable. Therefore, the only case left to study is when $|\lambda|=1$ and $\lambda$ is not a root of unity.

	In \cite{HY83}, Herman and Yoccoz studied the power series over a complete non-Archimedean field $K$ of characteristic $0$, and fully solved the problem on their linearizability by showing that every power series $f(z)=z\left(\lambda +\sum_{i=1}^\infty a_iz^i \right)$ over $K$ with  $\lambda$ not a root of unity and $|\lambda|=1$ is linearizable. Rivera-Letelier in \cite{RL03} gave a different proof of this result in a more general setting.
	Due to Herman and Yoccoz's work, there has
been an increasing interest in the non-Archimedean analogue of complex
dynamics, see e.g. \cite{Her87}, \cite{a-v1}, \cite{a-v2}, \cite{Lub94}, \cite{hsi}, \cite{ben}, \cite{RL03}, \cite{R03}, \cite{bez},  \cite{Lin04}, \cite{Lin10}.

	However, Herman and Yoccoz's  method cannot be generalized to positive characteristic fields since $\lambda$ does not satisfy the Siegel's condition in this case. In fact, Herman conjectured the opposite in \cite[p 147]{Her87} that most power series in this case are non-linearizable. In  \cite{Lin04} and \cite{Lin10}, Lindahl proved that  two specific families of polynomials are non-linearizable, which supports Herman's conjecture. We state Lindahl's result after fixing the following notations. 
	\begin{notation}
		Let $p>0$ be some prime number, and $\calK$ a complete non-Archimedean field of characteristic $p$. 
	\end{notation}
		\begin{theorem}[\cite{Lin04}, Theorem~2.3(2); \cite{Lin10}, Theorem~2]\label{Lindahl}
	Assume that $p\geq 3$, and that $\lambda\in \calK$ with $\lambda$ not a root of unity and $|1-\lambda|<1$. 
		Then every polynomial of the form
		\begin{enumerate}
			\item  $\lambda z+a_1z^2\in \calK[z]$ with $a_1\neq 0$, or
			\item $\lambda z+a_{p}z^{p+1}\in \calK[z]$ with $a_p\neq 0$ 
		\end{enumerate}
		is non-linearizable. 
	\end{theorem}

	  On the other hand, Lindahl proved a  family of linearizable power series over $\calK$ in \cite{Lin10}.  More precisely, he proved
	  \begin{theorem}[\cite{Lin10}, Theorem~3]\label{Lindahl1}
	  		Assume $\lambda\in \calK$ satisfies that $\lambda$ is not a root of unity and $|1-\lambda|<1$.  Then every power series of the form $\lambda z+\sum_{p|i} a_{i-1}z^{i}\in \calK[\![z]\!]$ is linearizable. 
	  \end{theorem}

	Based on Theorems~\ref{Lindahl} and~\ref{Lindahl1}, Lindahl conjectured that
	\begin{conjecture}[\cite{Lin04}, Conjecture~2.2]\label{Lin}
		Assume $\lambda\in \calK$ satisfies that $\lambda$ is not a root of unity and $|1-\lambda|<1$.  A polynomial of the form $\lambda z+\sum_{i=2}^n a_{i-1}z^{i}\in \calK[z]$ is linearizable if and only if $a_i=0$ for all $i\geq 2$ such that $p\nmid i$. 
	\end{conjecture}
	
	Note that Lindahl's method can only deal with the families of polynomials with a dominant $a_i$ such that $\frac{|a_i|}{i}>>\frac{|a_j|}{j}$ for all $j\neq i$. Hence, it cannot fully settle a family of polynomials with more than three terms, e.g. the family of the cubic polynomials $\{\lambda z+a_1z^2+a_2z^{3}\;|\; a_2, a_3\in \calK\}$.
	Due to this limitation, to attack Conjecture~\ref{Lin}, people needs to come up with some new ideas. This is the motivation behind our paper.

	

%
%
%
%
%
%
%
	
In this paper, we study general locally analytic (at $z=0$)  power series $f$ over $\calK$.  
  Our main contribution is proving a sufficient condition (see Criterion~\ref{c:key criterion} in \S\ref{sec:non-line-crit}) for such $f$ to be non-linearizable. The following theorems are applications of this criterion on three families of polynomials. 
	
%
%
%


%
%
%

	\begin{theorem}\label{example1}    
		 Assume that $\lambda\in \calK$ with $\lambda$ not a root of unity and $|1-\lambda|<1$. Then for $n\geq 1$ the polynomial $\lambda z+a_n z^{n+1}\in\mathcal{K}[z]$ is linearizable if and only if either $p\;|\; n+1$ or $a_n=0$.
	\end{theorem}

Note that when combined with Theorem~\ref{Lindahl1} and under the hypothesis $|1-\lambda|<1$, our Theorem~\ref{Lindahl1} fully solves the linearization problem for polynomials of the form $\lambda z+a_{n-1}z^{n}$ for any $n\geq 2$. In particular, it covers all the cases in Theorem~\ref{Lindahl}.


Beyond these cases, it is natural to ask for the conditions of polynomials of the form $z(\lambda +a_{i}z^i +a_{j}z^j)$ to be linearizable or vise versa. In particular, since the family $\{\lambda z +a_1z^2+a_{p-1}z^p\;|\; a_1, a_{p-1}\in \calK\}$ is a mixture case of 
Theorems~\ref{Lindahl}(1) and~\ref{Lindahl1}, fully understanding its linearizable behavior becomes especially important; and this motivates us to have the following result. 
	\begin{theorem}\label{Thm:p-1}
	 Assume $p \geqslant 5$, $\lambda\in \calK$ with $\lambda$ not a root of unity and $|1-\lambda|<1$.	Then $\lambda z+a_1 z^{2}+a_{p-1}z^{p}\in\mathcal{K}[z]$ is linearizable if and only if $a_1= 0$. 
	\end{theorem}

It is worth mentioning that Theorem~\ref{Thm:p-1} is an unconditional result over a family of polynomials with three terms, and hence it cannot be fully settled by Lindahl's method.  We also obtain the following conditional result on the family of cubic polynomials. 
	
	\begin{theorem}\label{example}     Assume $p \geqslant 5$, $\lambda\in \calK$  with $\lambda$ not a root of unity and $|1-\lambda|<1$. Then $\lambda z+a_1 z^{2}+a_2z^{3}\in\mathcal{K}[z]$ is non-linearizable if 
	$a_1 \neq 0$, $\left|1-\frac{a_2}{a_1^{2}}\right| \geqslant 1$ and $\left|1-\frac{a_2}{a_1^{2}}\right|\neq \frac{1}{|1-\lambda|}$.
		\end{theorem}
		
%
%
%
The above three theorems all support Lindahl's conjecture and have shown the strength of Criterion~\ref{c:key criterion}. On the other hand, note that to prove a polynomial non-linearizable, it is enough to show that it satisfies one of the infinite conditions in our criterion, i.e. $k$-dominant for some integer $k\geq 1$. Thus, we believe that our criterion has a lot of potentials to prove the non-linearizability of many other families of polynomials over $\calK$, and even possibly  to fully solve Lindahl's conjecture. 
In particular, we believe that by a more detailed calculation we  are able to prove each cubic polynomials with $a_1\neq 0$ is $k$-dominant for some $k\geq 1$, and hence non-linearizable.

	\subsection*{Acknowledgment}
	The author would like to express his deepest appreciation to Professor Juan Rivera-Letelier for his massive help and giving a concise version of the main criterion.

	\section{Preliminaries}
	\label{sec:preliminaries}
		
Throughout the rest of this paper we fix a prime number~$p>0$ and a complete non-Archimedean field $\calK$ of characteristic $p$.
Denote by~$\OK \= \{ x \in \cK \mid  |x|\leq 1\}$ the ring of integers of $\cK$, by~$\mK \= \{ x \in \cK \mid |x| < 1 \}$ the maximal ideal of~$\OK$ and by~$\Kres \= \OK / \mK$ the residual field of~$\cK$. 
For~$x$ in~$\OK$, denote by~$\tilde x$ its reduction in~$\Kres$.
We fix $\lambda\in\calO_\calK$ with $\lambda$ not a root of unity and $ |1-\lambda|<1$; and a locally analytic (at $z=0$) power series
\begin{displaymath}
	f(z) = z\left(\lambda + \sum_{n = 1}^{\infty} a_n z^n\right) \in \calK [\![z]\!].
\end{displaymath}


\begin{notation}\label{n:mu and o}
\noindent	\begin{enumerate}
		\item 	 Let $\mu:=\lambda-1$, and define $\mu$-adic valuation on $\calK$ as follows:
		for every $x\in \calK$ we put $$\v(x):=\frac{\ln|x|}{\ln |\mu|}.$$
		
			Note that for every integer $s\geq 1$ we have
		\begin{equation}\label{eq:s}
			\valmu(1-\lambda^{s})=p^{\vp(s)}.
		\end{equation}
		
		\item  Let  $u:=\gcd(i\;|\; a_i\neq 0) $ and $\tau:=\vp(u)$. 
	\end{enumerate}
\end{notation} 	

	\begin{definition}
		We call $f$  \emph{locally linearizable at $z=0$} (or \emph{linearizable} for short) if there exists a locally analytic (at $z=0$) power series $h(z)=z\left(\sum\limits_{n=0}^\infty b_n z^n\right)\in\calK [\![z]\!]$ such that $h\circ f\circ h^{-1}(z)=\lambda z$.
	\end{definition}

	\begin{notation}\label{binom}
		\noindent		\begin{enumerate}

		\item Let $\NN=\{0,1,\dots\}$ be the set of natural numbers.
		\item	Set~$a_0:= \lambda$.
		\item 	Given an integer~$s \ge 0$ and an $(s+1)$-tuple $\ualpha = (\alpha_0, \ldots, \alpha_s)$ in~$\NN^{s + 1}$, we put
		\begin{displaymath}
		| \ualpha | := \sum_{i = 0}^s \alpha_i,\quad
		\| \ualpha \| := \sum_{i = 1}^s i \alpha_i,\quad \ua^{\ualpha} := \prod_{i\in \{0,\dots,s\}, a_i\neq 0} a_i^{\alpha_i},
		\end{displaymath}
		where $a_0,\dots,a_s$ are first $(s+1)$-st coefficients of $f$;
		and denote by~$\binom{|\ualpha|}{\ualpha}_p$ the image of the multinomial coefficient~$\binom{|\ualpha|}{\alpha_0, \ldots, \alpha_s}$ under the composite map $$\ZZ\to\Fp \xhookrightarrow{} \calK.$$
		\item For every~$r$ in~$\{0, \ldots, s - 1\}$, we put
		\begin{multline}\label{sy1}
		I(r, s): = \big\{ \ualpha \in \NN^{s - r + 1}\;\big|\;\alpha_i=0\ \textrm{for any}\ 0\leq i \leq s-r\textrm{~with~}u\nmid i,
		 \ |\ualpha | = r + 1, \| \ualpha \| = s - r \big\},
		\end{multline}
	and
		\begin{equation}
		\label{e:recurrence-coefficient}
		\Phi(r, s)
		:=
		\frac{1}{\lambda(1 - \lambda^s)} \sum_{\ualpha \in I(r, s)} \binom{r + 1}{\ualpha}_p \ua^{\ualpha}\in \calK.
		\end{equation}
	
	Note that if $I(r,s)=\emptyset$, then $\Phi(r,s)=0$. 
	\end{enumerate}

	\end{notation}

		\begin{proposition}\label{definition of bk}
		The power series $f(z)$ is formally conjugate to $\lambda z$ by a unique formal power series $h(z)=z\sum\limits_{n=0}^\infty b_nz^n\in \calK[\![z]\!]$ with $h'(0)=1$.
		
		Moreover, the sequence $\{b_n\}$ satisfies the inductive relations:
		\begin{itemize}
			\item $b_0=1$.
			\item For every $n\geq 1$, 	\begin{equation}\label{bs}
				b_n = \sum_{\ell = 0}^{n - 1} b_\ell \Phi(\ell, n).
			\end{equation}
		\end{itemize}
		
	\end{proposition}
\begin{proof}
	By
	\cite[equations (6) and (7)]{Lin04}, 	\begin{equation*}
		b_n = \sum_{\ell = 0}^{n - 1} b_\ell 	\frac{1}{\lambda(1 - \lambda^s)} \sum_{\substack{\ualpha\in \NN^{s-r+1}\\|\ualpha|=r+1,\ ||\ualpha||=s-r }} \binom{r + 1}{\ualpha}_p \ua^{\ualpha}.
	\end{equation*}

	Note that for $$\ualpha\in \left\{\NN^{s - r + 1}\;\big|\; |\ualpha | = r + 1, \| \ualpha \| = s - r \right\}$$ with some $0\leq i\leq s-r+1$ such that $u\nmid i$ and $\alpha_i\neq 0$, we have $\ua^{\ualpha}=0$.  This proves 	$$\Phi(r,s)=\frac{1}{\lambda(1 - \lambda^s)} \sum_{\substack{\ualpha\in \NN^{s-r+1}\\|\ualpha|=r+1,\ ||\ualpha||=s-r }} \binom{r + 1}{\ualpha}_p \ua^{\ualpha},$$ and completes the proof.
\end{proof}
Our next target is to write $b_n$ explicitly (see Proposition~\ref{l:coefficient-formula2}). Before doing that, we introduce some notations. 

\begin{notation}\label{d:Phibeta}
	Given an increasing finite sequence of integers $\ubeta=(\beta_0,\dots,\beta_L)$ with $L\geq 1$, we put 	$m(\ubeta)
	:=
	L - 1 $
	to be the number of middle terms in $\ubeta$; and 
	\begin{equation*}
	\Phi(\ubeta)
	:=
	\prod_{j = 0}^{m(\ubeta)} \Phi(u\beta_j, u\beta_{j + 1}).
	\end{equation*}
	
	Given any integers $0\leq r<s$, we set \begin{equation*}
		S_\infty(r,s):=\bigsqcup_{t=1}^{s-r}\big\{\ubeta:=(\beta_0,\beta_1,\dots,\beta_t)\in \NN^{t+1}\;\big|\; 
		\beta_0=r,\ \beta_t=s\textrm{~and~} \beta_i<\beta_{i+1}\big\},
	\end{equation*}
	i.e.  $S_\infty(r, s)$ is the set of all finite increasing sequences of integers~$\ubeta$ such that \begin{center}
		$m(\ubeta)\leq s-r-1$ and
		$r = \beta_0 < \beta_1 < \ldots < \beta_{m(\ubeta)+1} = s.$
	\end{center}
	
\end{notation}

\begin{notation}\label{d:Sk,phik}
	Given any integers $0\leq r < s$ and a number $k\in\NN \cup \{ \infty \}$, we put
	\begin{equation*}
	S_k(r, s)
	 :=
	\left\{ \ubeta \in S_{\infty}(r, s)\;\big|\;p^k \nmid \beta_j \text{ for all } j \in \{1, \ldots, m(\ubeta) \} \right\},
	\end{equation*}
and 
\begin{equation}\label{eq:1}
	\phi_k(r, s)
	 :=
	\sum_{\ubeta \in S_k(r, s)} \Phi(\ubeta).
\end{equation}
\end{notation}

Note that in the case~$k = 0$, we have by definition
\begin{equation}
\label{eq:6}
S_0(r, s) = \{ (r, s) \}
\text{ and }
\phi_0(r, s)
=
\Phi(ur, us).
\end{equation}

The following is an inductive relation of $\phi_{\bullet}$.

\begin{lemma}
	\label{l:coefficient-formula1}
	Given any integers $0\leq r <s$ and $k'\in\NN \cup \{ \infty \}$,  we have for every $k\in\{0, \ldots, \min\{\vp(r),\vp(s),k'\} \}$,
	\begin{displaymath}
	\phi_{k'}(r,s)
	=
	\sum_{\ubeta \in S_{k'-k}(r p^{-k}, s p^{-k})} \prod_{j = 0}^{m(\ubeta)} \phi_{k}(\beta_j p^{k}, \beta_{j + 1} p^{k}).
	\end{displaymath}
	
		Note that the choice of $k$ makes both $r p^{-k}$ and  $s p^{-k}$ integers. 
\end{lemma}

\begin{proof}
	By regrouping the summands in~\eqref{eq:1}, we have
	\begin{equation}\label{new5}
	\phi_{k'}(r,s)
	=	\sum_{\ubeta \in S_{k'}(r, s)} \Phi(\ubeta)
	=\sum_{\ubeta\in S_{k'-k}(r p^{-k}, s p^{-k})}
	\prod_{j=0}^{m(\ubeta)} \left(\sum_{\underline{\gamma_j}\in S_{k}(\beta_jp^{k},\beta_{j+1}p^{k})}
	\Phi(\underline{\gamma_j})\right).
	\end{equation}
	Note that for every $\ubeta\in S_{k'-k}(r p^{-k}, s p^{-k})$ and $0\leq j\leq m(\ubeta)$, the last term in  \eqref{new5} satisfies
$$	\sum_{\underline{\gamma_j}\in S_{k}(\beta_jp^{k},\beta_{j+1}p^{k})}
	\Phi(\underline{\gamma_j})=\phi_{k}(\beta_jp^{k},\beta_{j+1}p^{k}).$$
	We complete the proof.
\end{proof}

\begin{proposition}
	\label{l:coefficient-formula2}
	For every integer~$n \ge 1$, 
	\begin{enumerate}
		\item if $u\nmid n$, then $b_n=0$;
		\item if  $u\;|\; n$, then for every integer $k\in \{0, \ldots, \vp(n/u) \}$, we have
	$$	b_n=	\sum\limits_{\ubeta \in S_\infty(0, p^{-k}n/u)} \prod\limits_{j = 0}^{m(\ubeta)} \phi_k(\beta_j p^k, \beta_{j + 1} p^k).$$
	\end{enumerate} 
\end{proposition}
\begin{proof}
	We first show that for any integers $0\leq r<s$, if $u\nmid s-r$, then 
	\begin{equation}\label{21q}
		\Phi(r,s)=0.
	\end{equation} 
It is enough to show that  in this case $I(r,s)=\emptyset$. This can be proved by contradiction. Suppose otherwise, let $\ualpha\in I(r,s)$, then we have
	$$u\;|\;||\ualpha||=s-r,$$ a contradiction.

	We next prove by induction that \begin{equation}\label{22b}
		b_n =\begin{cases}
			\phi_{\infty}(0, n/u)
			& \textrm{if~}u\;|\; n,\\
			0& \textrm{if~}u\nmid n.
		\end{cases} 
	\end{equation}
Consider \eqref{bs} that $b_1=\Phi(0,1).$
	For $n=1,$ if $u=1$, then $b_1=\Phi(0,1)=\phi_\infty(0,1)$; if $u\geq 2$, by \eqref{21q}, we have $b_1=\Phi(0,1)=0$. In both cases, we prove \eqref{22b} for $n=1$.
	
	Now for any given $n\geq 2$, assume \eqref{22b} holds for every $1\leq \ell\leq n-1$.
	Note that for any $0\leq \ell\leq n-1$ with $u\nmid n-\ell$, by \eqref{21q} we have \begin{equation}\label{22w}
		\Phi(\ell, n)=0. 
	\end{equation}

	For $n$ such that $u\nmid n$,  all $0\leq \ell\leq n-1$ with $u\;|\; n-\ell$ satisfy $u\nmid \ell$. Hence, from our induction hypothesis on $\ell$, this condition implies $b_\ell=0$; and by \eqref{bs}, we have $b_n=0$. 
	
	If $u\;|\; n$, combining \eqref{bs} and \eqref{22w} with our induction hypothesis, we have $$b_n=\sum_{i=0}^{n/u} b_{ui}\Phi(ui,n)=\Phi(0,n)+\sum_{i=1}^{n/u} \phi_{\infty}(0,i)\Phi(ui,n)=\phi_\infty(0,n/u).$$
	This completes the induction. 
	
Combining \eqref{22b} and Lemma~\ref{l:coefficient-formula1} with $r = 0$, $s = n/u$ and $k' = \infty$, we complete the proof of this proposition.
\end{proof}

\section{Non-linearizability criterion}
\label{sec:non-liner-crit}

The goal of this section is to formulate a criterion for power series over $\calK$ to be non-linearizable that  we use in~\S\ref{sec:application}, \S\ref{sec:application1}; and give its proof under the assumption of the ``congruence'' property that is proved in~\S\ref{sec:congruence-property}.
In~\S\ref{sec:non-line-crit} we state the non-linearizability criterion, as well as the congruence and ``dominance'' properties.
In~\S\ref{sec:proof-criterion} we prove the non-linearization criterion assuming the congruence and dominance properties, and in~\S\ref{sec:dominance-property} we prove the dominance property assuming the congruence property.


\subsection{Non-linearizability criterion}
\label{sec:non-line-crit}

In this subsection, we state the non-linearizability criterion.

\begin{notation}
  \label{d:slope}
  Given any integers $r,s,k$ such that $0\leq r< s$ and $k\geq 0$,  we put
  \begin{align*}
  \psi_k(r, s) & \= \phi_k(r, s) \frac{1 - \lambda^{us}}{(1 - \lambda^{p^{k+\tau}})\lambda^{us - 1}}.
  \end{align*}
Note that if $\vp(s)=k$, then   \begin{equation}\label{new6}
	\val_{\mu}(\phi_k(r, s))=\val_{\mu}(\psi_k(r, s)).
\end{equation}
 
For $r, s$ further satisfying $p^k|r$ and $p^k|s$, we put
  \begin{equation}\label{a3}
  M_k(r, s)  \= \frac{\valmu (\psi_k(r, s))}{s - r},
  \end{equation}
and moreover set
  \begin{equation}\label{new7}
    M_k \= \inf \big\{ M_k(r, s)\;\big|\; p^k\textrm{-divisible integers~} 0\leq r< s\big\}.
 \end{equation}
\end{notation}

 Note that by \eqref{eq:6}, we have 
\begin{equation}
\label{eq:4}    
\psi_0(r, s)
=
\Phi(ur, us) \frac{1 - \lambda^{us}}{\lambda^{us - 1} (1 - \lambda^{p^{\tau}})}.
\end{equation}

\begin{definition}
	\label{d:dominance}
	Given an integer~$k \ge 1$, we say that $f$ is \emph{$k$-dominant} if
	\begin{equation*}\label{new9}
	   \min_{d\in \{1,\dots,p-1\}} \left\{M_{k}(0,dp^k)\right\}
	\le
        \min_{ \ell\in\{0,\dots,k-1\}} \left\{M_\ell \right\} - p^{\tau-1}.
	\end{equation*}
\end{definition}

\begin{customcri}{$\bigstar$}[Non-linearizability criterion]\label{c:key criterion}
	If $f$ is $k$-dominant for some integer $k \ge 1$, then it is non-linearizable.
\end{customcri}

In the next section we deduce this criterion from the following propositions.

\begin{proposition}[Congruence property]
	\label{p:congruence-property}
	Given any integer $k\geq 1$ and any $p^k$-divisible integers $r,s,m$ such that $0\leq r<s$ and $m\geq -r$, 
	 we have
	\begin{displaymath} 
	\valmu\left(\psi_k\left(r+m, s+m\right) - \psi_k\left(r, s\right)\right)
	>
	(s-r)  \min_{ \ell\in\{0,\dots,k-1\}} \left\{M_\ell\right\}  - p^{k+\tau-1}.
	\end{displaymath}  
\end{proposition}

\begin{proposition}[Dominance property]
	\label{p:dominance-property}
	If $f$ is $k$-dominant for some integer $k \ge 1$, then  it is $k'$-dominant
	for all integers $k'\geq k$.
\end{proposition}

\subsection{Proof of the non-linearizability criterion assuming Propositions~\ref{p:congruence-property} and~\ref{p:dominance-property}}
\label{sec:proof-criterion}

The proof of the non-linearizability criterion is at the end of this subsection, after several lemmas.

Combining  \eqref{eq:s} with Notation~\ref{d:Sk,phik}, for any $k\geq 0$ and  $p^k$-divisible integers $0\leq r<s$, we have
\begin{multline}
  \label{eq:2}
  M_k(r, s)=\frac{\valmu (\psi_k(r, s))}{s - r}=\frac{\valmu (\phi_k(r, s))+\valmu (1-\lambda^{us})-\valmu (1-\lambda^{p^{k+\tau}})}{s - r}\\
	=
	\frac{\valmu (\phi_k(r, s))}{s - r} +  \frac{p^{\vp(s)+\tau} - p^{k+\tau}}{s - r}.
\end{multline}

\begin{lemma}
  \label{l:18}
  Given any $k\geq 0$ and $p^k$-divisible integers $0\leq r<s$, we have
  \begin{displaymath}
    \v (\phi_{k+1}(r, s))
    \ge
    (s - r)  M_{k} + p^{k+\tau} - p^{\vp (s)+\tau }.
              \end{displaymath}
\end{lemma}
\begin{proof}
	Taking $k':=k+1$ in Lemma~\ref{l:coefficient-formula1},
we have
		\begin{equation}\label{1}
	\phi_{k+1}(r,s)
	=
	\sum_{\ubeta \in S_{1}(r/p^k , s/p^k)} \prod_{j = 0}^{m(\ubeta)} \phi_{k}(\beta_j p^{k}, \beta_{j + 1} p^{k}).
	\end{equation}
	
	Note that for any~$\ubeta\in S_1(r/p^k, s/p^k)$ and $1\leq j\leq m(\ubeta)$, we have $p\nmid \beta_j$. Therefore, by \eqref{eq:2} we have 
	$$\v\left(\phi_{k}(\beta_jp^{k},\beta_{j+1}p^{k})\right)=(\beta_{j+1}p^{k}-\beta_jp^{k})M_{k}\left(\beta_jp^{k},\beta_{j+1}p^{k}\right) \textrm{~for~} 0\leq j\leq m(\ubeta)-1$$
	and  $$  \v\left(\phi_{k}(\beta_{m(\ubeta)}p^{k},s)\right)=   (s-\beta_{m(\ubeta)}p^{k})M_{k}\left(\beta_{m(\ubeta)}p^{k},s\right)+ p^{k+\tau} - p^{\vp(s)+\tau }. $$

These two equalities above imply
	\begin{align*}
          &\v\left(\prod_{j=0}^{m(\ubeta)} \phi_{k}(\beta_jp^{k},\beta_{j+1}p^{k})\right)
\\=&      \sum_{j=0}^{m(\ubeta) - 1} (\beta_{j+1}p^{k}-\beta_jp^{k})M_{k}\left(\beta_jp^{k},\beta_{j+1}p^{k}\right)
              +(s-\beta_{m(\ubeta)}p^{k})M_{k}(\beta_{m(\ubeta)}p^{k},s) + p^{k+\tau} - p^{\vp(s)+\tau} \\
              \geq&
          (s-r)M_{k}  + p^{k+\tau} - p^{\vp(s)+\tau},
      \end{align*}
 where $M_k$ is defined in \eqref{new7}. 
  
From the strong triangular inequality, by plugging these inequalities in \eqref{1} we complete the proof.
\end{proof}

\begin{lemma}\label{l:M(r,s)<minM}
  Assume that~$f$ is $k$-dominant for some $k\geq 1$.
  Then 
\begin{enumerate}
	\item   
given any $p^k$-divisible integers $0\leq r<s$ with $s - r \ge p^{k+1}$, we have  $M_k(r, s)> M_k$.
\item
  $
    M_k
    =
   \min\limits_{d\in \{1,\dots,p-1\}}\left\{M_{k}(0,dp^k)\right\}.$
\end{enumerate}\end{lemma}
\begin{proof}
	
	(1) For any $p^k$-divisible $0\leq r<s$, by~\eqref{eq:2} and Lemma~\ref{l:18} with $k+1$ replaced by $k$, we have
	\begin{multline}\label{eq:A1}
	M_k(r,s)
	=
	\frac{\v(\phi_k(r,s))}{s - r} + \frac{p^{\vp(s)+\tau} - p^{k+\tau}}{s-r}\\
	\ge
	\frac{(s - r)M_{k-1}  + p^{k-1+\tau} - p^{\vp (s)+\tau}}{s - r}
	+ \frac{p^{\vp(s)+\tau} - p^{k+\tau}}{s-r}
	=
	M_{k-1}-\frac{p^{k+\tau}-p^{k-1+\tau}}{s-r}. 
	\end{multline}
	From our assumption that  $s-r\geq p^{k+1}$ and that $f$ is $k$-dominant, the last term in \eqref{eq:A1} satisfies  $$    M_{k-1}-\frac{p^{k+\tau}-p^{k-1+\tau}}{s-r}
	>M_{k - 1} - p^{\tau-1}
	\geq \min_{d\in \{1,\dots,p-1\}} \left\{M_{k}(0,dp^k)\right\}.$$
	This implies
	\begin{equation}\label{b1}
	M_k(r, s)
	>
	\min_{d\in \{1,\dots,p-1\}} \left\{M_{k}(0,dp^k)\right\}\geq M_k.
	\end{equation}
	
	(2) By \eqref{b1}, it is enough to show that 	for every $p^k$-divisible integers $0\leq r<s$ such that $s - r \le (p - 1)p^k$,  we have
\begin{equation}\label{b2}
		M_k(r, s)
	\geq 
	\min_{d\in \{1,\dots,p-1\}} \left\{M_{k}(0, dp^k)\right\}.
\end{equation}

For integers $r$, $s$ with this condition, we consider 
  \begin{equation}\label{new8}
      \v (\psi_{k}(r, s))
  \\
  \geq \min\big\{\v (\psi_{k}(0, s-r)), \v \big(\psi_{k}(r, s)-\psi_{k}(0, s-r)\big)\big\},
  \end{equation}
  and clearly have
    $$   \v \left(\psi_{k}(0, s-r)\right)= (s-r)M_{k}(0, s-r)  
       \geq (s-r)\min_{d\in \{1,\dots,p-1\}}  \left\{M_k(0, dp^k)\right\}.$$
       
       On the other hand, combining Proposition~\ref{p:congruence-property} with our hypothesis that~$f$ is $k$-dominant, we have
       \begin{multline*}
       \v \Big(\psi_{k}(r, s)-\psi_{k}(0, s-r)\Big)\geq  (s-r) \left(\min_{ \ell\in\{0,\dots,k-1\}} \left\{M_\ell\right\}\right) - p^{k +\tau-1}
       \\
     \geq (s-r)\left(  \min_{d\in \{1,\dots,p-1\}}  \left\{M_k(0, dp^k)\right\}  +p^{\tau-1}\right)  - p^{k +\tau-1}
     \\ \geq (s-r)\min_{d\in \{1,\dots,p-1\}}  \left\{M_k(0, dp^k)\right\}.
       \end{multline*}

       These two inequalities above with \eqref{new8} imply
	\begin{equation*}
	\v (\psi_{k}(r, s))
	\geq  (s-r)\min_{d\in \{1,\dots,p-1\}}  \left\{M_k(0, dp^k)\right\},
	\end{equation*}   
and further \eqref{b2}. This completes the proof. 
\end{proof}

\begin{proof}[Proof of Criterion~\ref{c:key criterion} (Assuming Propositions~\ref{p:congruence-property} and \ref{p:dominance-property})]
Let $k'$ be an arbitrary integer such that $k'\geq k$.  By Proposition~\ref{p:dominance-property}, the series
   $f$ is $k'$-dominant, and hence
    	\begin{multline*}\label{e:inductive inequality}
 \min_{d \in \{1, \ldots, p - 1 \}} \left\{M_{k'}(0, dp^{k'})\right\}
 \le
 \min_{\ell \in \{0, \ldots, k' - 1 \}} \left\{M_\ell\right\}- p^{\tau-1}
\\
 \le M_{k'-1}- p^{\tau-1}
 \leq 
 \min_{d \in \{1, \ldots, p - 1 \}} \left\{M_{k'-1}(0, dp^{k'-1})\right\} - p^{\tau-1}.
 \end{multline*}

%
%
%
%
%
                      By induction, we have
			\begin{equation}\label{e:inductive inequality1}
                          \min_{d \in \{1, \ldots, p - 1 \}} \left\{M_{k'}(0, dp^{k'})\right\}
                          \le
                          \min_{d \in \{1, \ldots, p - 1 \}} \left\{M_{k}(0, dp^k)\right\}  - (k'-k)p^{\tau-1}.
		\end{equation}
	
			By Lemma~\ref{l:M(r,s)<minM}(2), we  have 
		$$M_{k'}=\min_{d \in \{1, \ldots, p - 1 \}} \left\{M_{k'}(0, dp^{k'})\right\}.$$

 Hence, there is the smallest $d_{k'}\in\{1,\dots,p - 1\}$ such that
\begin{displaymath}
M_{k'}(0, d_{k'}p^{k'})
=M_{k'}.
\end{displaymath}

Replacing $k$ by $k'$ and $n$ by  $ud_{k'}p^{k'}$ in Proposition~\ref{l:coefficient-formula2}(2), we have
                \begin{equation}\label{1a}
	b_{ud_{k'}p^{k'}}=\sum_{\ubeta\in S_\infty(0,d_{k'})}\prod_{j=0}^{m(\ubeta)}\phi_{k'}(\beta_jp^{k'},\beta_{j+1}p^{k'}).
      \end{equation}
      
      Now we estimate $\v(b_{ud_{k'}p^{k'}})$ by studying each summand in \eqref{1a}.
      
      For any $\ubeta\in S_\infty(0,d_{k'})$ and $0\leq j\leq m(\ubeta)$, we have
      \begin{equation}\label{new1}
      1\leq \beta_{j+1}-\beta_j\leq d_{k'}\leq p-1,
      \end{equation} 
      and hence $p\nmid \beta_{j+1}$.
      
      This implies
      \begin{multline}\label{new2}
           \v (\phi_{k'}(\beta_{j}p^{k'}, \beta_{j+1}p^{k'})) \\= \v (\psi_{k'}(\beta_{j}p^{k'}, \beta_{j+1}p^{k'}))+p^{\vp(\beta_{j+1}p^{k'})+\tau}-p^{k'+\tau}
           =\v (\psi_{k'}(\beta_{j}p^{k'}, \beta_{j+1}p^{k'})).
      \end{multline}
      Consider 
             \begin{multline}\label{new3}
 \v (\psi_{k'}(\beta_{j}p^{k'}, \beta_{j+1}p^{k'}))\\
  \geq
      \min\left\{\v \left(\psi_{k'}(0, (\beta_{j+1}-\beta_{j})p^{k'})\right), \v \left(\psi_{k'}(\beta_{j}p^{k'}, \beta_{j+1}p^{k'})-\psi_{k'}(0, (\beta_{j+1}-\beta_{j})p^{k'})\right)\right\}.
      \end{multline}

    Combining \eqref{new1} with the smallest assumption on $d_{k'}$, we have 
\begin{multline}\label{new4}
	\v \left(\psi_{k'}(0, (\beta_{j+1}-\beta_{j})p^{k'})\right)= (\beta_{j+1}-\beta_{j})p^{k'}M_{k}(0,(\beta_{j+1}-\beta_j)p^{k'})
	\\  \geq (\beta_{j+1}-\beta_{j})p^{k'}M_{k'},     
\end{multline}
with equality if and only if $\beta_{j}=0$ and $\beta_{j+1}=d_{k'}$.

      Since $f$ is $k'$-dominant, by Proposition~\ref{p:congruence-property}, we have
      \begin{multline*}
        \v \left(\psi_{k'}(\beta_{j}p^{k'}, \beta_{j+1}p^{k'})-\psi_{k'}(0, (\beta_{j+1}-\beta_{j})p^{k'})\right)\\
              \begin{aligned}
      	>&(\beta_{j+1}-\beta_{j}) p^{k'}  \left(\min_{ \ell\in\{0,\dots,k'-1\}} \left\{M_\ell\right\} \right) - p^{k' - 1+\tau}\\
      	\geq &	(\beta_{j+1}-\beta_{j}) p^{k'}  \left( \min_{d \in \{1, \ldots, p - 1 \}} \left\{M_{k'}(0, dp^{k'})\right\}+p^{\tau-1}\right)-p^{k'-1+\tau}\\
      	\geq &(\beta_{j+1}-\beta_{j}) p^{k'}   M_{k'}.
     \end{aligned}
     \end{multline*}
     Plugging it with \eqref{new4} into \eqref{new3}, for every~$\ubeta$ in~$S_\infty(0,d_{k'})$, we have 
	 $$ \v\left(\prod_{j=0}^{m(\ubeta)}\phi_{k'}(\beta_jp^{k'},\beta_{j+1}p^{k'})\right)
         \ge
         d_{k'}p^{k'}M_{k'},$$
	 with equality if and only if~$\ubeta = (0,d_{k'})$.
	 
	 Therefore, we have
	\begin{equation*}
          \frac{\v(b_{ud_{k'}p^{k'}})}{ud_{k'}p^{k'}}
          =
          \frac{d_{k'}p^{k'}M_{k'}}{ud_{k'}p^{k'}}
          =\frac{M_{k'}}{u}.
	\end{equation*}
	Note that the above equation holds for all $k'\geq k$, by \eqref{e:inductive inequality1} we have 
	$$\frac{\v(b_{ud_{k'}p^{k'}})}{ud_{k'}p^{k'}}    \xrightarrow{k'\to \infty} -\infty.$$
	
This proves that the power series~$h(z)$ is not locally analytic at $z=0$, and consequently that~$f$ is non-linearizable.	
\end{proof}

\subsection{Proof of the Proposition~\ref{p:dominance-property} assuming Proposition~\ref{p:congruence-property}}
\label{sec:dominance-property}
Throughout this subsection we fix an integer~$k \ge 1$ and assume that~$f$ is $k$-dominant.  By Lemma~\ref{l:M(r,s)<minM}(2), we can find the greatest $d^*\in \{1, \ldots, p - 1 \}$ such that~$M_k(0, d^*p^k) = M_k$.


We will prove 
  $$\min_{d\in \{1,\dots,d^*\}} \left\{M_{k}(0, dp^{k})\right\}
\le
\min_{ \ell\in\{0,\dots,k-1\}} \left\{M_\ell \right\} - p^{\tau-1},
$$
which implies that $f$ is $(k+1)$-dominant directly.

To achieve this, we need to estimate 
$\v(\phi_{k + 1}(0, dp^{k+1}))$ for every $d\in \{1,\dots,d^*\}$. 
Our strategy is to construct a sequence $\{c_m\}_{m\in \ZZ}\subset \calK$ recursively which satisfies that
\begin{itemize}
	\item  $c_{dp}$ is close to $\mu^{p^{k+1+\tau}-p^{k+\tau}}\phi_{k + 1}(0, dp^{k+1})$ for $d\in \{1,\dots,d^*\}$ (see Lemma~\ref{claim 2}); and
	\item $\v(c_{dp})$ is relatively easy to be estimated (see Lemma~\ref{l:claim}).
\end{itemize}  

%

%
%

        For every $m\geq 1$, we put~$$A_m \=
        \frac{(1 - \lambda^{p^{\vp(um) + k}}) \lambda^{ump^k - 1}}{1 - \lambda^{ump^k}}\quad\textrm{and}\quad f_m \= \psi_k ( 0, mp^k ).$$
        
   Note that for every $m\geq 1$ we have
   \begin{equation}\label{fn}
   \v(f_m)=\v\left(\psi_k ( 0, mp^k )\right)\geq  m p^kM_k;
   \end{equation} 
	and that by \eqref{eq:s},  \begin{equation}\label{n1}
	\v(A_m) = 0.
	\end{equation}

              We define the sequence $\{{c}_m \}_{m \in \ZZ}$ recursively by 
                \[c_m \=\begin{cases}
                0&\textrm{ for\ }m \le - 1,\\
                1& \textrm{ for\ }m =0,\\
                A_m \sum\limits_{d = 1}^{d^*} f_d  c_{m-d}&\textrm{ for\ }m\geq 1.
                \end{cases}\]

	Combining \eqref{fn} with \eqref{n1},  for every~$m \ge 0$  by induction we have $\v(c_m) \ge  m p^kM_k$.

\begin{lemma}\label{claim 2}
	Suppose that for every $1\leq d\leq  d^* - 1$ we have
	\begin{equation}
	\label{eq:17}
	\v \left( \phi_{k + 1}(0, d p^{k + 1}) \right)
	>
	 d p^{k + 1}M_k + p^{k+\tau} - p^{k + 1+\tau}.
	\end{equation}
	Then for every~$m$ in~$\{1, \ldots, d^* p \}$, we have
	\begin{equation}\label{a2}
	\v\left(c_{m} - \mu^{p^{\vp(um) + k} - p^{k+\tau}} \phi_{k + 1}(0, mp^k) \right)
	>
	mp^kM_k.
	\end{equation}
\end{lemma}

\begin{proof}
	We first prove that	for every $m\geq 1$, \begin{multline}\label{a5}
	\v\left(	\mu^{p^{\vp(um) + k} - p^{k+\tau}} \phi_{k + 1} (0, m p^k)- A_mf_m - \sum_{\substack{w \in \{1, \ldots, m - 1 \} \\ p \nmid w}}A_m \phi_{k + 1}(0, w p^k) f_{m-w}\right)\\
	>mp^k M_k.
\end{multline}

	
By Lemma~\ref{l:coefficient-formula1} with~$k'$, $r$ and $s$ replaced by $k+1$, $0$ and $mp^k$, we have 
	\begin{equation}\label{new 3}
 \phi_{k + 1} (0, m p^k)
= 
\sum_{\ubeta\in S_{1}(0, m)}
\prod_{j=0}^{m(\ubeta)} \left(\sum_{\underline{\gamma_j}\in S_{k}(\beta_jp^{k},\beta_{j+1}p^{k})}
\Phi(\underline{\gamma_j})\right).
				\end{equation}
			We decompose $S_1(0,m)$ into a union of disjoint sets as
			\begin{equation}\label{new 2}
			S_1(0,m)=\{(0,m)\}\cup \bigcup_{\substack{w\in \{1,\dots,m-1\}\\p\nmid w}}\{\ubeta\in S_{1}(0, m)
				\;|\;	\beta_{m(\ubeta)}=w\}.
			\end{equation}

			For each $1\leq w\leq m-1$ such that $p\nmid w$, we have
		\begin{multline}\label{new 1}
			\sum_{\substack{\ubeta\in S_{1}(0, m)\\\beta_{m(\ubeta)}=w}}
			\prod_{j=0}^{m(\ubeta)} \left(\sum_{\underline{\gamma_j}\in S_{k}(\beta_jp^{k},\beta_{j+1}p^{k})}
			\Phi(\underline{\gamma_j})\right)=
			\\
			\left(
			\sum_{\underline{\gamma}\in S_{k}(w p^{k},mp^{k})}
			\Phi(\underline{\gamma})\right)\times
			\left(
				\sum_{\ubeta'\in S_{1}(0, w)} 
			\prod_{j=0}^{m(\ubeta')} \left(\sum_{\underline{\gamma_j}\in S_{k}(\beta'_jp^{k},\beta'_{j+1}p^{k})}
			\Phi(\underline{\gamma_j})\right)\right).
		\end{multline}
			
%
	
Note that for every $0\leq w\leq m-1$ we have \begin{multline*}
	\sum_{\underline{\gamma}\in S_{k}(w p^{k},mp^{k})}
	\Phi(\underline{\gamma})=\phi_{k}(w p^k,mp^k)
	=\frac{(1 - \lambda^{p^{k+\tau}})\lambda^{ump^k - 1}}{1 - \lambda^{ump^k}}\psi_{k}(w p^k,mp^k)
	\\
= A_m \mu^{p^{k+\tau}-p^{\vp(um)+k}}\psi_{k}(w p^k,mp^k),
	\end{multline*}
	and that
	by Lemma~\ref{l:coefficient-formula1}, for every $0\leq w\leq m-1$ with $p\nmid w$, 
$$	\sum_{\ubeta'\in S_{1}(0, w)} 	\prod_{j=0}^{m(\ubeta')} \left(\sum_{\underline{\gamma'_j}\in S_{k}(\beta'_jp^{k},\beta'_{j+1}p^{k})}
	\Phi(\underline{\gamma'_j})\right)=\phi_{k + 1}(0, w p^k).$$

	Plugging them into \eqref{new 1} and considering the decomposition of $S_1(0,m)$ in \eqref{new 2}, we simplify \eqref{new 3} to 
\begin{multline}	\label{eq:15}
	 \phi_{k + 1} (0, m p^k)
	 \\= A_m \mu^{p^{k+\tau}-p^{\vp(um)+k}}\left(\psi_k(0, m p^k) + \sum_{\substack{w \in \{1, \ldots, m - 1 \} \\ p \nmid w}} \phi_{k + 1}(0, w p^k) \psi_{k}(w p^k,mp^k)\right),
\end{multline} 
and immediately obtain
	\begin{multline}\label{a10}
			\mu^{p^{\vp(um)+k}-p^{k+\tau}} \phi_{k + 1} (0, m p^k)- A_mf_{m} - \sum_{\substack{w \in \{1, \ldots, m - 1 \} \\ p \nmid w}}A_m \phi_{k + 1}(0, w p^k)f_{m-w}
		\\
	=	A_m\left(\psi_{k} (0, mp^k) -f_m\right)- \sum_{\substack{w \in \{1, \ldots, m - 1 \} \\ p \nmid w}}A_m \phi_{k + 1}(0, w p^k) \left(\psi_{k} (w p^k, mp^k)-f_{m-w}\right).
\end{multline}

%

%
%
%
	
Note that from \eqref{n1}, $\v(A_m)=1$.  Therefore, to prove \eqref{a5}, it is enough to show \begin{equation}\label{21d}
	\valmu\left(\psi_{k} (0, mp^k) - f_{m}\right)  > 	 mp^kM_k;
\end{equation}
and for every $w\in \{1,\dots,m\}$ with $p\nmid w$,	\begin{equation}\label{n2}
	\v \left(\phi_{k + 1}(0, w p^k) \left(\psi_{k}(w p^k,mp^k)-f_{m-w}\right)\right)>mp^k M_k.
\end{equation}

By Proposition~\ref{p:congruence-property} and the definition of the sequence $\{f_m\}$, for every $w \in \{0, \ldots, m - 1 \}$  we have 
	\begin{equation*}
	\valmu\left(\psi_{k} (w p^k, mp^k) - f_{m-w}\right)
	>
 (m-w)p^k	  \min_{ \ell\in\{0,\dots,k-1\}} \left\{M_\ell\right\} - p^{k - 1+\tau}.
	\end{equation*}
	
	Combined with our assumption that $f$ is $k$-dominant, this inequality implies	
	\begin{multline}\label{21e}
		\valmu\left(\psi_{k} (w p^k, mp^k) - f_{m-w}\right) 
\\	>
 (m-w)p^k	\left(  \min_{d\in \{1,\dots,p-1\}} \left\{M_{k}(0,dp^k)\right\}+p^{\tau-1}\right) - p^{k - 1+\tau}
 \geq 	 (m-w)p^kM_k.
		\end{multline}
	Taking $w=0$, we obtain \eqref{21d}.
	
For any $w\in \{1,\dots,m\}$ with $p\nmid w$, by Lemma~\ref{l:18} we have 
	$$	\v \left(\phi_{k + 1}(0, w p^k)\right) \geq wp^k M_k.$$
Combined with \eqref{21e}, this proves \eqref{n2}.

%


	Now we prove \eqref{a2} by induction. 	
Note first that by \eqref{a5} for any given $m\geq 1$,  \eqref{a2} is equivalent to \begin{equation}\label{20p}
	\v\left(	c_m-A_mf_m-\sum_{\substack{w \in \{1, \dots, m-1 \}\\p\nmid w}}A_m \phi_{k + 1}(0, w p^k) f_{m-w}\right)
	>
	m p^kM_k.
\end{equation}

For~$m = 1$, we have
$	c_1 = A_1 f_1$. This implies \eqref{20p} and hence \eqref{a2}. 
	
	For any integer $m$ in $\{2,\dots,d^*p\}$ suppose that \eqref{a2} holds for every~$m'$ in~$\{1, \ldots, m -1 \}$.
	As noted above, we just need to prove \eqref{20p} instead. Since $c_m$ is defined piece-wisely, we first prove that the induction works for $2\leq m\leq d^*$; and then prove it for bigger $ m$'s. 
	
	For 
	$m$ in~$\{2, \ldots, d^*\}$, since $m\leq d^*\leq p-1$, we can get rid of $p\nmid w$ condition in  \eqref{20p} for such $m$.
	

Combining it with the recursive definition of $\{c_m\}$, we have 
\begin{multline}\label{b4}
	c_m-A_mf_m - \sum_{w \in \{1, \dots, m-1 \}}A_m \phi_{k + 1}(0, w p^k) f_{m-w}
	\\
	= \sum_{w \in \{1, \dots, m-1 \}}A_mf_{m-w}  (c_w-\phi_{k + 1}(0, w p^k)).
\end{multline}

Note that 	$\v\left(\mu^{p^{\vp(uw)+k}-p^{k+\tau}}\right)=1$ for every $w\geq 1$.
Combining \eqref{fn} with our induction hypothesis, for every $w \in \{1, \dots, m-1 \}$ we have 
	\begin{equation}\label{20q}
		A_mf_{m-w}  (c_w-\phi_{k + 1}(0, w p^k))>mp^k	M_k.
	\end{equation}
Combined with \eqref{b4},  this implies \eqref{20p} and hence \eqref{a2} for $m$.

	Now we focus on $m\in \{d^*+1,\dots, d^*p\}$. Note that we still assume \eqref{a2} holds for every~$m'$ in~$\{1, \ldots, m -1 \}$.
	By Lemma~\ref{l:M(r,s)<minM}(1), for every $n\geq p$ we have 
	\begin{equation*}
	\v(f_n)=\v(\psi_k(0, np^k) )=np^kM_k(0,np^k)>  np^kM_k.
	\end{equation*}
	
	Together with the maximum choice of $d^*$, the range of $w$ can be enlarged to $n\geq d^*+1$, i.e.
	\begin{equation}\label{n3}
	\v(f_n) >  n p^kM_k \quad \textrm{for every~} n\geq d^*+1.
	\end{equation}
	
	Combining it with  Lemma~\ref{l:18},  
	for every~$1\le w \le m-d^* - 1$ such that $p\nmid w$ we have
	\begin{equation}\label{a4}
	\v\left(\phi_{k + 1}(0, wp^k) f_{m-w}\right)> mp^kM_k.
	\end{equation}
	
	By \eqref{n3} with $n=m$ and \eqref{a4}, the inequality  \eqref{20p} is equivalent to
	\begin{equation}\label{20w}
	\v\left(c_m- \sum_{\substack{w \in \{m-d^*, \dots, m-1 \}\\p\nmid w}}A_m \phi_{k + 1}(0, w p^k) f_{m-w}\right)\\
	>mp^k M_k.
	\end{equation}


Similar to \eqref{20q}, by \eqref{fn} and induction hypothesis on $w\leq m-1$, for every $w\in \{m-d^*, \dots, m-1 \}$ with $p\nmid w$ we have 
\begin{equation}\label{20t}
	\v\left(A_m  \left(c_w-\phi_{k + 1}(0, w p^k)\right)f_{m-w}\right)>mp^kM_k.
\end{equation}

If $\{m-d^*, \dots, m-1 \}$ does not contain $p$-divisible integers,  from the recursive definition of $\{c_m\}$, we have 
\begin{multline}\label{eq:19}
	c_m-\sum_{\substack{w \in \{m-d^*, \dots, m-1 \}\\p\nmid w}}A_m \phi_{k + 1}(0, w p^k) f_{m-w}
	\\
	= \sum_{\substack{w \in \{m-d^*, \dots, m-1 \}}}A_m  \left(c_w-\phi_{k + 1}(0, w p^k)\right)f_{m-w}.
\end{multline}

Combined with \eqref{20t}, this equality implies
\eqref{20w}. 

If $\{m-d^*, \dots, m-1 \}$ contains $p$-divisible integers, since $d^*\leq p-1$, it has to be unique; and we denote it by $dp$.
	
		In this case, the recursive definition of $\{c_m\}$ gives us one extra term than \eqref{eq:19}, i.e. 
	\begin{multline}\label{20o}
		c_m-\sum_{\substack{w \in \{m-d^*, \dots, m-1 \}\\p\nmid w}}A_m \phi_{k + 1}(0, w p^k) f_{m-w}
		\\
		= \sum_{\substack{w \in \{m-d^*, \dots, m-1 \}\\p\nmid w}}A_m  (c_w-\phi_{k + 1}(0, w p^k))f_{m-w}+A_m c_{dp}f_{m-dp}.
	\end{multline}

	Note that
	$$d \leq \left\lfloor\frac{m-1}{p}\right\rfloor\leq\left\lfloor \frac{d^*p-1}{p}\right\rfloor=d^*-1,$$ so by our assumption~\eqref{eq:17} we have
	\begin{displaymath}
	\v \left( \mu^{p^{k + 1+\tau} - p^{k+\tau}} \phi_{k + 1}(0, dp^{k + 1}) \right)
	>
	 d p^{k + 1}M_k.
	\end{displaymath}
	Together with our induction hypothesis on $dp<m$, this implies
	\begin{multline*}
	\v(c_{dp})\ge\\
	\min \left\{ \v\left(\mu^{p^{k + 1+\tau} - p^{k+\tau}} \phi_{k + 1}(0, dp^{k+1})\right), \v\left(c_{dp} - \mu^{p^{k + 1+\tau} - p^{k+\tau}} \phi_{k + 1}(0, dp^{k+1})\right) \right\} 
\\>	dp^{k+1}M_k,
	\end{multline*}
	and therefore
	\begin{equation}\label{20u}
	\v ( c_{dp}f_{m-dp})
	>
 m p^k	M_k,
	\end{equation}
	when combined with \eqref{fn}.
	
	Combining \eqref{20t}, \eqref{20o} and \eqref{20u}, we obtain \eqref{20w} again. Note that \eqref{20w} implies \eqref{20p} and further \eqref{a2} for such given $m$. We complete the induction. 
\end{proof}
\begin{lemma}
  \label{l:claim}
There is an integer $d$ in~$\{1, \ldots, d^* \}$ such that $\v( c_{d p})=dp^{k+1}M_k$.
\end{lemma}

\begin{proof}
Note that $\v(c_{d p}) \ge  d p^{k+1}M_k$ for all $1\leq d\leq d^*$.  We can assume by contradiction that for every~$d$ in~$\{1, \ldots, d^* \}$ we have~$\v( c_{d p}) > d p^{k+1}M_k$.

 Enlarging~$\cK$ if necessary, assume there is~$\zeta$ in~$\cK$ such that~$\v(\zeta) =  p^kM_k$.
  Note that for each~$d$ in~$\{1, \ldots, d^* \}$, we have that~$\whf_d \= f_d \zeta^{-d}$ is in~$\OK$ and that~$|\whf_{d^*}| = 1$.
  On the other hand, for every integer~$j \ge 1$, we have that~$\whc_j \= c_j \zeta^{-j}$ is also in~$\OK$.
  Denote by~$\xi_j$ the reduction of~$\whc_j$.
  Therefore, our assumption is equivalent to that for every~$d$ in~$\{1, \ldots, d^* \}$ we have~$\xi_{dp} = 0$.
  
  Note that for every integer~$m \ge 1$ the matrix
\begin{displaymath}
  N_m
  \=
\left(\begin{array}{@{}c|c@{}}
\begin{matrix}
0 \\
\vdots \\
0
\end{matrix}&
I
\\ \hline
A_m \whf_{d^*}
&\begin{matrix}
A_m \whf_{d^*-1}, & \cdots &, A_m \whf_1
\end{matrix}
\end{array}\right)
\end{displaymath}
has coefficients in $\OK$ and that its reduction~$\tN_m$ is invertible and satisfies 
\begin{equation}
  \label{eq:20}
  \begin{bmatrix}  \xi_{m - (d^* - 1)},\ldots, \xi_{m-1} ,  \xi_{m} \end{bmatrix}^T
  \\ =
  \tN_m
\begin{bmatrix}  \xi_{m-d^*} , \ldots,  \xi_{m-2},  \xi_{m-1} \end{bmatrix}^T.
\end{equation}
Noting that for an integer~$m \ge 1$ not divisible by~$p$ we have~$\tA_m = \tA_{m + p}$, we obtain~$\tN_m = \tN_{m + p}$.
Note also that for every~$d$ in~$\{1, \ldots, d^* \}$ we have by assumption~$\xi_{dp} = 0$. Combined with \eqref{eq:20}, this implies 
$$\hat f_{d^*}\xi_{dp - d^*}+\cdots+\hat f_1\xi_{dp - 1}=0,$$
and hence
\begin{displaymath}
  \begin{bmatrix}
    \xi_{dp - (d^* - 1)},\ldots, \xi_{dp - 1} ,  \xi_{d p} \end{bmatrix}^T
  \\ =
  \tN_p
\begin{bmatrix}  \xi_{dp - d^*} , \ldots,  \xi_{dp - 2},  \xi_{dp - 1} \end{bmatrix}^T.
\end{displaymath}
Therefore, if we put
$$N \= N_p N_{p - 1} \ldots N_1, $$
then for every~$d$ in~$\{1, \ldots, d^* \}$ we have
\begin{equation}
  \label{eq:21}
  \begin{bmatrix} \xi_{dp - (d^* - 1)}, \ldots , \xi_{d p - 1},  \xi_{d p} \end{bmatrix}^T
=
\tN^d \begin{bmatrix} 0 , \ldots , 0 ,1 \end{bmatrix}^T.
\end{equation}
Moreover, the matrix~$\tN$ is invertible, so its characteristic polynomial~$P(x):=\sum\limits_{i=0}^{d^*}\tau_ix^i\in \widetilde{\calK}[x]$ satisfies~$\tau_0 \neq 0$.
From~\eqref{eq:21} we obtain
$$ \sum_{j=0}^{d^*} \tau_j\begin{bmatrix} \xi_{j p+d^*-1}, \ldots , \xi_{j p+1},  \xi_{j p} \end{bmatrix}^T
=
P(\tN)\begin{bmatrix}0 , \ldots , 0 ,1 \end{bmatrix}^T
=
\begin{bmatrix}0 , \ldots , 0 ,0 \end{bmatrix}^T,$$
and therefore $\sum\limits_{j=0}^{d^*}\tau_j  \xi_{j p}=0$.
However, since $\tau_0 \neq 0$, $\xi_0 = 1$ and by our assumption $\xi_{d p} = 0$ for every $d$ in~$\{1, \ldots, d^* \}$, we have  $$\sum\limits_{j=0}^{d^*}\tau_j  \xi_{j p}=\tau_0\xi_0\neq 0.$$
We thus obtain a contradiction that proves the lemma.
\end{proof}

\begin{proof}[Proof of Proposition~\ref{p:dominance-property}]
  By induction, it is enough to show that~$f$ is $(k+1)$-dominant.

  We claim that for some~$d'$ in~$\{1, \ldots, d^* \}$ we have
  \begin{equation}
    \label{eq:18}
    \v \left( \phi_{k + 1}(0, d' p^{k + 1}) \right)
    =
     d' p^{k + 1} M_k + p^{k+\tau} - p^{k + 1+\tau}.    
  \end{equation}
  If this is not the case, then by Lemma~\ref{l:18} with $r = 0$ and $d'=dp^{k+1}$, we obtain~\eqref{eq:17} for every~$d$ in~$\{1, \ldots, d^* \}$.
  Then for every~$d$ in~$\{1, \ldots, d^* \}$ by Lemma~\ref{claim 2} with~$m = dp$, we have 
    \begin{multline*}
      \v\left(c_{dp} \right)\\
      \ge
      \min \left\{ \v \left( \mu^{p^{k + 1+\tau} - p^{k+\tau}} \phi_{k + 1}(0, dp^{k + 1}) \right), \v \left(c_{dp} - \mu^{p^{k + 1+\tau} - p^{k+\tau}} \phi_{k + 1}(0, dp^{k + 1}) \right) \right\}>
      M_kdp^k.
\end{multline*}

This contradicts Lemma~\ref{l:claim} and hence proves that~\eqref{eq:18} holds from some~$d'$ in~$\{1, \ldots, d^* \}$.
Since~$d' \le d^* \le p - 1$, from \eqref{eq:2} we have 
	\begin{displaymath}
          M_{k+1}(0,d'p^{k+1})
          =
          \frac{\v\left(\phi_{k+1}(0,d' p^{k+1})\right)}{d' p^{k+1}}
          =
                M_k-\frac{p^{\tau} - p^{\tau-1}}{d'}
        \le
        M_k- p^{\tau-1}.
	\end{displaymath}
      This proves that~$f$ is $(k + 1)$-dominant and completes the proof of the proposition.
\end{proof}

\section{Proof of the congruence property}
\label{sec:congruence-property}

In this section we give the proof of Proposition~\ref{p:congruence-property}.
It depends on two estimates, which are stated as Lemmas~\ref{level0} and~\ref{key lemma 1} in~\S\ref{sec:key-estimate} and~\S\ref{sec:inductive-estimate}, respectively.
The proof of Proposition~\ref{p:congruence-property} is given in~\S\ref{sec:proof-prop-refp:c}.
\subsection{Key estimate}
\label{sec:key-estimate}
The purpose of this section is to prove the key estimate (see Lemma~\ref{level0}), which is used in the proof of Proposition~\ref{p:congruence-property}.

We start with the following lemma that is also used in~\S\ref{sec:application}, \S\ref{sec:application1}.
Recall that we define $I(\bullet,\bullet)$ in Notation~\ref{binom}.

\begin{lemma}
	\label{l:0-th}
	We have
	\begin{displaymath}
	M_0 = \inf \left\{ \frac{\valmu(a_{ui}) - p^\tau}{i}\;\Big|\; i \ge 1 \right\}.
	\end{displaymath}
\end{lemma}
\begin{proof}
(1)	Since $f$ is locally analytic at $z=0$, the following infimum limit
$$	M_0' \= \inf \left\{ \frac{\valmu(a_{ui}) - p^\tau}{i}\;\Big|\; i \ge 1 \right\}$$
exists. 

	We first prove 
$	M_0
	\ge
	M_0'$.
	Given any integers $0\leq r<s$ and  $\ualpha$ in~$I(ur, us)$, we  have
	\begin{displaymath}
	\sum\limits_{i=1}^{s-r}\alpha_{ui}
	\ge
	\frac{\|\ualpha\|}{us-ur}
	=
	1,
	\end{displaymath}
	and therefore
	\begin{equation*}
	\valmu(\ua^{\ualpha})
	=
	\sum_{i = 1}^{s - r} \alpha_{ui} \valmu(a_{ui})
	\ge
	\sum_{i = 1}^{s - r} \alpha_{ui} (i M_0' + p^\tau)
	=
	(s - r) M_0' + p^\tau\sum\limits_{i=1}^{s-r}\alpha_{ui}
	\ge
	(s - r) M_0' + p^\tau.
	\end{equation*}

	This implies
	\begin{multline*}
	(s - r) M_0(r, s)
	=
	\valmu(\psi_0(r, s))
	=
	\valmu \left( \Phi(ur, us) \frac{1 - \lambda^{us}}{(1 - \lambda^{p^\tau})\lambda^{us-1}} \right)
	\\ =
	\valmu \left(\sum_{\ualpha \in I(ur, us)} \binom{r + 1}{\ualpha}_p \ua^{\ualpha} \right) - p^\tau
	\ge
	(s - r) M_0',
	\end{multline*}
	and proves~$M_0 \ge M_0'$.
	
	Now we prove $	M_0	\le	M_0'$. Note that for every $i\geq 1$, the set $I(0,ui)$ contains a unique $\ualpha:=(0,\dots, 0,1)\in \NN^{ui+1}$. Hence, we have \begin{equation}\label{22a}
		\Phi(0,ui)=\frac{a_{ui}}{\lambda(1-\lambda^{ui})}.
	\end{equation}

	Assume that there exists $i'\geq 1$ such that $\frac{\v(a_{ui'})-p^\tau}{i'}=M_0'$. Set $$i_0:=\min\left\{i\geq 1\;\Big|\; \frac{\v(a_{ui})-p^\tau}{i}=M_0'\right\}.$$
	Then by \eqref{22a}, we have $$\v(\psi_0(0,i_{0}))=\v\left(\Phi(0,ui_0)\frac{1-\lambda^{ui_0}}{\lambda^{ui_0-1}(1-\lambda^{p^\tau})}\right)=\v(a_{ui_0})-p^\tau,$$ and hence $M_0\leq \frac{\v(a_{ui_0})-p^\tau}{i_0}=M_0'.$
	
%
%
%
%
	If no such $i'$ exists, then there is a strictly increasing sequence $\{i_n\}_{n=1}^\infty$ such that 
	for every $n\geq 1$, $$\frac{\v(a_{ui_n})-p^\tau}{i_n}<\frac{\v(a_{ui})-p^\tau}{i} \textrm{~for all~} 1\leq i< i_n; $$and
\begin{equation}\label{20a}
		\lim_{n\to \infty}\frac{\v(a_{ui_n})-p^\tau}{i_n}=M_0'. 
\end{equation}
	
	Similar to the argument above on $i_0$, for every $n\geq 1$ we have
$$\v(\psi_0(0,i_{n}))=\v(a_{ui_n})-p^\tau,$$ and hence $$M_0\leq\frac{\v(\psi_0(0,i_{n}))}{i_n}= \frac{\v(a_{ni_n})-p^\tau}{i_n}.$$
Combined with 
	\eqref{20a}, we obtain $M_0\leq M_0'$.
\end{proof}

\begin{lemma}\label{fl}
	For any integers $0\leq r < s$ and any $(s-r+1)$-tuple $\ualpha$ in~$I(r,s)$, if
	\begin{equation*}
	\left\lfloor \frac{r+1}{p^j}\right\rfloor
	> \sum_{i=0}^{s-r}\left\lfloor \frac{\alpha_i}{p^j}\right\rfloor\quad \textrm{for some integer~} j \ge 1,
	\end{equation*}
	then 
	$ \binom{r+1}{\ualpha}_p=0.$
\end{lemma}
\begin{proof}
	We make the following calculation:
	\begin{multline*}
	\vp \left( \binom{r+1}{\ualpha} \right)
	=
	\vp\left((r+1)!\right)-\sum_{i=0}^{s-r}\vp(\alpha_i!)
	=
	\sum_{k=1}^{\infty} \left( \left\lfloor \frac{r+1}{p^k}\right\rfloor - \sum_{i=0}^{s-r}\left\lfloor \frac{\alpha_i}{p^k}\right\rfloor\right)
	\\ \ge
	\left\lfloor \frac{r + 1}{p^j} \right\rfloor - \sum_{i=0}^{s-r} \left\lfloor \frac{\alpha_i}{p^j}\right\rfloor
	\ge
	1,
	\end{multline*}
	which completes the proof.
\end{proof}

\begin{lemma}
  \label{l:psi(r',s')-psi(r,s)}
  Given any integers $r,s, m$ with~$0\leq r<s$ and $m\geq-r$, we have
  \begin{multline*}
		\psi_0(r+m, s+m)-\psi_0(r, s) 
		\\ =
		\frac{1}{(1 - \lambda^{p^{\tau}}) \lambda^{us+um}} \sum_{\substack{\ualpha\in I(ur+um,us+um)\\ |\ualpha'| - \alpha_0' \ge p^{\vp(m)+\tau}}} \binom{ur+um+1}{\ualpha'}_p \underline{a}^{\ualpha'}
                -
                \frac{1}{(1 - \lambda^{p^{\tau}}) \lambda^{us}} \sum_{\substack{\ualpha\in I(ur,us)\\ |\ualpha| - \alpha_0 \ge p^{\vp(m)+\tau}}} \binom{ur+1}{\ualpha}_p \underline{a}^{\ualpha}. 
		\end{multline*}
\end{lemma}

\begin{proof}
  Put
  \begin{itemize}
  	\item 	 $q := p^{\vp(um)}=p^{\vp(m)+\tau}$,
  	\item    $  I := \left\{\ualpha\in I(ur,us)\;\big|\;|\ualpha| - \alpha_0 < q \right\}$ and   $I':=\left\{\ualpha'\in I(ur+um,us+um)\;\big|\; |\ualpha'| - \alpha'_0 < q \right\}.$
  \end{itemize}
       Note that from~\eqref{eq:4} and the definition of~$\Phi(ur, us)$ and~$\Phi(ur+um, us+um)$, the
       lemma is equivalent to
       \begin{equation}
         \label{eq:7}
         \frac{1}{\lambda^{us+um}} \sum_{\ualpha' \in I'} \binom{ur+um + 1}{\ualpha'}_p \ua^{\ualpha'}
         =
         \frac{1}{\lambda^{us}} \sum_{\ualpha \in I} \binom{ur + 1}{\ualpha}_p \ua^{\ualpha}.
       \end{equation}
  
       To prove this, note first that in the case~$q = p^\tau$ we have~$I = I' = \emptyset$, so~\eqref{eq:7} holds trivially
       in this case.
       From now on, we assume~$q \ge p^{1+\tau}$. Without loss of generality, we assume further that $m> 0$.
      Then for each~$\ualpha$ in~$I$ the multi-index
         \begin{displaymath}
           G(\ualpha) \= (\alpha_0 +um,0,\dots,0, \alpha_u, 0,\dots,0,\alpha_{2u},\dots,\alpha_{u(s - r)})
         \end{displaymath}
         is in~$I'$.
         Moreover, the map~$G \colon I \to I'$ so defined is injective.
         On the other hand, for every~$\ualpha'$ in~$I' \backslash G(I)$ we have $\alpha'_0 < um$ and~$\sum_{i = 1}^{r - s} \alpha_{ui}' < q$, and consequently
                 \begin{displaymath} 
                   \left\lfloor \frac{ur+um + 1}{q}\right\rfloor
                   >
                   \left\lfloor \frac{\alpha_0'}{q}\right\rfloor
                   =
                   \sum_{i=0}^{s-r} \left\lfloor \frac{\alpha_{ui}'}{q}\right\rfloor.
                 \end{displaymath}
                 So by Lemma~\ref{fl} we have~$\binom{ur+um+1}{\ualpha'}_p=0$, which implies 
          \begin{multline*}
              \frac{1}{\lambda^{us+um}} \sum_{\ualpha' \in I'} \binom{ur+um + 1}{\ualpha'}_p \ua^{\ualpha'}
              \\=
                 \frac{1}{\lambda^{us+um}} \sum_{\ualpha' \in G(I)} \binom{ur+um + 1}{\ualpha'}_p \ua^{\ualpha'}+ \frac{1}{\lambda^{us+um} }\sum_{\ualpha' \in I'\backslash G(I)} \binom{ur+um + 1}{\ualpha'}_p \ua^{\ualpha'}
                 \\=   
                 \frac{1}{\lambda^{us+um}} \sum_{\ualpha' \in G(I)} \binom{ur+um + 1}{\ualpha'}_p \ua^{\ualpha'}.
          \end{multline*}
  
                 Therefore, to prove~\eqref{eq:7} it is sufficient to show that every $\underline{\alpha}\in I$ we have
        \begin{equation}
          \label{eq:5}
          \binom{ur+1}{\ualpha}_p = \binom{ur+um+1}{G(\ualpha)}_p.
        \end{equation}
        Put~$\ualpha' \= G(\ualpha)$.
        From~$\sum\limits_{i=1}^{s-r}\alpha_{ui} < q$ and the definition of~$q$, we have
        \begin{multline}
          \label{fl4}
          \left\lfloor \frac{ur+1}{q}\right\rfloor-\sum_{i=0}^{s-r}\left\lfloor \frac{\alpha_{ui}}{q}\right\rfloor
          =
  \left\lfloor \frac{ur+1}{q}\right\rfloor - \left\lfloor \frac{\alpha_0}{q}\right\rfloor
  =
  \left\lfloor \frac{ur+um + 1}{q}\right\rfloor - \left\lfloor \frac{\alpha_0 + um}{q}\right\rfloor
  \\
  =
  \left\lfloor \frac{ur+um+ 1}{q}\right\rfloor-\sum_{i=0}^{s-r} \left\lfloor \frac{\alpha_{ui}'}{q}\right\rfloor.
\end{multline}
If this number is strictly positive, then by Lemma~\ref{fl} we have
$$ \binom{ur+1}{\ualpha}_p = \binom{ur+um+1}{\ualpha'}_p=0, $$
and therefore~\eqref{eq:5}.
Thus, to complete the proof of the lemma it remains to prove~\eqref{eq:5} in the case where the number~\eqref{fl4} is equal to zero.
In this case we have~$\left\lfloor \frac{ur+1}{q}\right\rfloor = \left\lfloor \frac{\alpha_0}{q}\right\rfloor$, which implies that for every~$j$ in~$\{\alpha_0 + 1, \ldots, r + 1 \}$ we have~$q \nmid j$.
Therefore, the rational number
\begin{displaymath}
  \rho
  \=
  \frac{\binom{r+m + 1}{\ualpha'}}{\binom{r + 1}{\ualpha}}
  =
  \prod_{j = \alpha_0 + 1}^{r + 1} \frac{j + m}{j}
  =
  \prod_{j = \alpha_0 + 1}^{r + 1} \left(1 + \frac{m}{j} \right)
  \end{displaymath}
  satisfies~$\vp(\rho - 1) \ge 1$.
  This implies~\eqref{eq:5} and completes the proof of the lemma.
\end{proof}

\begin{lemma}\label{level0}
 Given any integers~$0\leq r<s$ and $m\geq-r$, we have
	\begin{enumerate}
		\item 
		$\v\left(\psi_0(r+m,s+m)-\psi_0(r,s)\right) \geq (s-r)M_0 + p^{\vp(m)+2\tau}-p^\tau$;
		\item
		$\v\left(\Phi(ur+um,us+um)-\Phi(ur,us)\right)
		\geq
		(s-r)M_0 + p^{\vp(m)+\tau} + p^{\tau} - p^{\vp(s)+\tau} - p^{\vp(s+m)+\tau}$.
	\end{enumerate}
\end{lemma}

\begin{proof}
  To prove~(1), note that by Lemma~\ref{l:0-th}, for every~$\ualpha\in I(ur,us)$ such that $|\ualpha|-\alpha_0\geq p^{\vp(m)+\tau}$  we have
  	\begin{multline*}
	\valmu(\ua^{\ualpha})
	=
	\sum_{i = 1}^{s - r} \alpha_{ui} \valmu(a_{ui})
	\ge
	\sum_{i = 1}^{s - r} \alpha_{ui} (iM_0 + p^\tau)
	\\
	=
	M_0 \| \ualpha \|/u + (| \ualpha | - \alpha_0)p^\tau\geq (s-r)M_0+p^{\vp(m)+2\tau}.
	\end{multline*}
	Similarly, for every~$\ualpha'\in I(r',s')$ such that $|\ualpha'|-\alpha'_0\geq p^{\vp(m)+\tau}$  we have
	\begin{equation*}
	\valmu(\ua^{\ualpha'})
\ge
 (s-r)M_0+p^{\vp(m)+2\tau}.
	\end{equation*}
	
Therefore, by Lemma~\ref{l:psi(r',s')-psi(r,s)} and $\v(\lambda)=0$, the above two inequalities imply
\begin{displaymath}
	\v(\psi_0(r+m,s+m) - \psi_0(r,s))
          \ge 
          (s-r)M_0 + p^{\vp(m)+2\tau} - p^\tau.    
	\end{displaymath}
	This proves~(1).
        To prove~(2), consider
	\begin{multline}\label{e:dif}
          \Phi(ur+um,us+um)-\Phi(ur,us)
          \\
          \begin{aligned}
	=&\frac{\lambda^{us+um-1}(1-\lambda^{p^\tau})}{1-\lambda^{us+um}}\left(\psi_0(r+m,s+m)-\psi_0(r,s)\right)+\frac{(1-\lambda^{us})\lambda^{um}}{1-\lambda^{us+um}}\Phi(ur,us)-\Phi(ur,us)\\
	=&\frac{\lambda^{us+um-1}(1-\lambda^{p^\tau})}{1-\lambda^{us+um}}\left(\psi_0(r+m,s+m)-\psi_0(r,s)\right)-\Phi(ur,us)\frac{1-\lambda^{um}}{1-\lambda^{us+um}}.\\
	\end{aligned}
	\end{multline}

	From \eqref{eq:6},
	\begin{equation*}\label{key 3}
	\v(\Phi(ur,us))=\v(\phi_0(r,s))\geq (s-r) M_0+ p^\tau - p^{\vp(s)+\tau},
	\end{equation*}
	and this implies
	\begin{equation*}\label{dif2}
          \v\left(\Phi(ur,us)\frac{1-\lambda^{um}}{1-\lambda^{us+um}}\right)
          \ge
         (s-r) M_0 +p^{\vp(m)+\tau} + p^{\tau}- p^{\vp(s)+\tau} - p^{\vp(s+m)+\tau}.
	\end{equation*}
	Combined with~\eqref{e:dif} and part~(1), this completes the proof of the lemma.
\end{proof}

\subsection{Inductive estimate}
\label{sec:inductive-estimate}

The purpose of this subsection is to prove Lemma~\ref{key lemma 1}, which is used in the proof of Proposition~\ref{p:congruence-property} in~\S\ref{sec:proof-prop-refp:c}.

We first introduce some notations.
For two finite sequences~$\ubeta$ and~$\ubeta'$ satisfying $\beta'_0=\beta_{m(\ubeta)+1}$, put
$$\ubeta\vee \ubeta':=(\beta_0,\dots,\beta_{m(\ubeta)},\beta'_0,\dots,\beta'_{m(\ubeta')+1}).$$
For any integers $0\leq r< s$, we denote by $S_\infty^*(r,s)$ the set of all sequences~$\uxi$ in~$S_\infty(r,s)$ that satisfy the following property.

If there is some~$\ell^*\in \{0, \ldots, m(\uxi) \}$ such that $\vp(\xi_{\ell^*})> \vp(\xi_{\ell^* + 1})$, then  
$\vp(\xi_{\ell})\geq \vp(\xi_{\ell + 1})$ for all~$\ell\in \{\ell^* + 1, \ldots, m(\uxi) \}$.

For each~$k$ in~$\NN$, we put $S_k^*(r, s):=S_\infty^*(r,s)\cap S_k(r, s)$.

\begin{lemma}\label{lemma 4.14}
  Given any $0\leq r<s$ and $k\in\NN \cup \{\infty\}$,   we have
  \begin{equation*}
  \phi_k(r,s)=\sum_{\uxi\in S_k^*(r,s)}\prod_{i=0}^{m(\uxi)}\phi_{\min\{\vp(\xi_i),\vp(\xi_{i+1})\}}(\xi_i,\xi_{i+1}).
  \end{equation*}

\end{lemma}

\begin{proof}
		Clearly, given any $\uxi\in S_k^*(r, s)$ and $\ugamma_i\in S_{\min\{\vp(\xi_i),\vp(\xi_{i+1})\}}(\xi_i,\xi_{i+1})$ for~$i\in \{0, \ldots, m(\uxi) \}$, we have 	\begin{equation*}
		\ubeta
		\=
		\bigvee\limits_{i=0}^{m(\uxi)} \ugamma_i\in S_k(r, s).
		\end{equation*}
		
		Therefore, to prove this lemma, it is enough to prove that for any $\ubeta\in S_k(r, s)$ there is a unique~$\uxi$ in~$S_k^*(r, s)$ and  a unique set of  sequences $\{\ugamma_i\}$ with $\ugamma_i\in S_{\min\{\vp(\xi_i),\vp(\xi_{i+1})\}}(\xi_i,\xi_{i+1})$ for $i$ in~$\{0, \ldots, m(\uxi) \}$ such that
	\begin{equation}
	\label{eq:3}
	\ubeta=
	\bigvee\limits_{i=0}^{m(\uxi)} \ugamma_i.
	\end{equation}
	
	We first prove the existence of such $\uxi$ and $\{\ugamma_i\}_{i=0}^{m(\uxi)}$.
	
	Put $\xi_0:=r.$ Let $\xi_1$ be the smallest number in $\{r+1,\dots,s\}$ such that $\xi_1\in \ubeta$ and $\vp(\xi_1)\geq \vp(r)$. Let $\xi_2$ be the smallest number in $\{\xi_1+1,\dots,s\}$ such that $\xi_2\in \ubeta$ and $\vp(\xi_2)\geq \vp(\xi_1)$. Keep this iteration until it stops. We denote this sequence by  $(\xi_0,\xi_1,\dots,\xi_t)$. 
	
	Put $\xi_0':=s$. Let $\xi_1'$ be the greatest number in $\{\xi_t,\dots,s-1\}$ such that $\xi_1'\in \ubeta$ and $\vp(\xi_1')\geq  \vp(s)$. Let $\xi_2'$ be the greatest number in $\{\xi_t,\dots,\xi_1'-1\}$ such that $\xi_2'\in \ubeta$ and $\vp(\xi_2')\geq \vp(\xi_1')$. Keep this iteration until it stops. Then we obtain an increasing sequence $(\xi_{t'}',\dots,\xi_0')$.
	From the choice of $\xi_t$ we have $\xi_t=\xi'_{t'}$, and $\uxi:=(\xi_0=r,\xi_1,\dots,\xi_t=\xi'_{t'},\dots,\xi_0'=s)\in S_k^*(r,s)$.  
	For every $0\leq i\leq m(\xi)$ let $\ugamma_i$ be the subsequence of $\ubeta$ which contains all $\beta\in \ubeta$ such that $  \xi_i\leq \beta\leq \xi_{i+1}$.  From the construction of $\uxi$ we have $\vp(\beta)<\min\{\vp(\xi_i),\vp(\xi_{i+1})\} $ for every $\beta\in \ugamma_i$ with $\beta\neq \xi_i$ or $\xi_{i+1}$. This implies 
$	\ugamma_i\in S_{\min\{\vp(\xi_i),\vp(\xi_{i+1})\}}(\xi_i,\xi_{i+1})$ and hence \eqref{eq:3}. 
	
	Now we prove the uniqueness. 
	
	Assume that for a given $\ubeta\in S_k(r,s)$ there are \begin{itemize}
		\item $\uxi$ and $\uxi'$ in $S_k^*(r,s)$;
		\item $ \ugamma_i\in S_{\min\{\vp(\xi_i),\vp(\xi_{i+1})\}}(\xi_{i},\xi_{i+1})$ for $i\in \{0,\dots,m(\uxi)\}$;
		\item $ \ugamma'_j\in S_{\min\{\vp(\xi'_j),\vp(\xi'_{j+1})\}}(\xi'_{j},\xi'_{j+1})$ for $j\in \{0,\dots,m(\uxi')\}$,
	\end{itemize} 
 such that $$	\ubeta= \bigvee\limits_{i=0}^{m(\uxi)} \ugamma_i=\bigvee\limits_{j=0}^{m(\uxi')} \ugamma_j'.$$
 
	Clearly, if $\uxi=\uxi'$, then the two tuples $(\uxi, \{ \ugamma_i\}_{i=0}^{m(\uxi)})$ and $(\uxi', \{ \ugamma_i'\}_{i=0}^{m(\uxi')})$ are identical. Now we assume that $\uxi\neq \uxi'$. Without loss of generality, we can set $i^*$ be the index in $(0,1,\dots,m(\uxi))$ such that 
	\begin{itemize}
		\item $\xi_i=\xi_i'$ for every $0\leq i\leq i^*-1$, 
		\item $\xi_{i^*}<\xi_{i^*}'.$
	\end{itemize}

From the relations $\xi'_{i^*-1}=\xi_{i^*-1}<\xi_{i^*}<\xi_{i^*}'$ and $\xi_{i^*}\in \ubeta$ we have
$\xi_{i^*}\in \ugamma'_{i^*-1}$, and consequently $\vp(\xi_{i^*})<\vp(\xi_{i^*}')$ and
\begin{equation}\label{l1}
\vp(\xi_{i^*})<\vp(\xi_{i^*-1}')=\vp(\xi_{i^*-1}).
\end{equation}

Since $\uxi\in S_k^*(r,s)$, the strict inequality \eqref{l1} implies $\vp(\xi_{i^*})\geq \vp(\xi_{i})$ for all $i\in \{i^*+1,\dots,m(\uxi)\}$, and hence $ \vp(\xi_{i^*})\geq \vp(\beta)$ for all $\beta\in \ubeta$ with $\beta\geq \xi_{i^*}$. This is a contradiction to  $\xi_{i^*}'>\xi_{i^*}$ and $\vp(\xi_{i^*})<\vp(\xi_{i^*}')$ and hence proves the uniqueness of such a representation.
      \end{proof}
\begin{lemma}\label{key lemma 1}
	Given any $0\leq r< s$ and $k \ge \min \{ \vp(r), \vp(s) \} + 1$, we have
	\begin{equation*}
	\v(\phi_{k}(r,s))
	\geq
	\min_{\ell \in \{0, \ldots, k - 1 \}} \left\{M_\ell\right\}(s-r) +p^{\vp(r)+\tau} - p^{\max\{k-1, \vp(r), \vp(s)\}+\tau}.
	\end{equation*}
\end{lemma}

      \begin{proof}
	By Lemma~\ref{lemma 4.14}, it is enough to show that for every~$\uxi$ in~$S_{k}^*(r,s)$ we have
        \begin{multline}
          \label{eq:10}
          \v\left(\prod_{i=0}^{m(\uxi)}\phi_{\min\{\vp(\xi_i),\vp(\xi_{i+1})\}}(\xi_i,\xi_{i+1})\right)
          \\ \ge
          \min\limits_{\ell \in \{0, \ldots, k - 1 \}} \left\{M_\ell\right\} (s-r) + p^{\vp(r)+\tau} - p^{\max\{k-1, \vp(r), \vp(s)\}+\tau}.
        \end{multline}
        Let~$i'$ be the smallest index in~$\{0, \ldots, m(\uxi) + 1 \}$ such that $\vp(\xi_{i'})$ reaches the maximal $p$-adic valuation among all terms in $\uxi$.
        Note that for every $0\leq i\leq i'-1$ we have $\vp(\xi_i)\leq \vp(\xi_{i+1})$; and 
        for every $i'\leq i\leq m(\uxi)$, $\vp(\xi_i)\geq  \vp(\xi_{i+1})$.
       Hence,           \begin{multline}\label{n6}
        \v\left(\prod_{i=0}^{m(\uxi)}\phi_{\min\{\vp(\xi_i),\vp(\xi_{i+1})\}}(\xi_i,\xi_{i+1})\right)
        \\ 
        \begin{aligned}
        & =
        \sum_{i=0}^{{i'}-1}\v\left(\phi_{\vp(\xi_i)}(\xi_i,\xi_{i+1})\right)+\sum_{i={i'}}^{m(\uxi)}\v\left(\phi_{\vp(\xi_{i+1})}(\xi_i,\xi_{i+1})\right)
        \\ &
        =
        \sum_{i=0}^{{i'}-1}\left((\xi_{i+1}-\xi_i)M_{\vp(\xi_i)}\left( \xi_i, \xi_{i+1}  \right)+p^{\vp(\xi_i)+\tau}-p^{\vp(\xi_{i+1})+\tau}\right) +
        \sum_{i={i'}}^{m(\uxi)}\left((\xi_{i+1}-\xi_i) M_{\vp(\xi_{i+1})}\left( \xi_i , \xi_{i+1}  \right)\right)
       \\ &
       =
       p^{\vp(r)+\tau}-p^{\vp(\xi_{i'})+\tau}+\sum_{i=0}^{{i'}-1}(\xi_{i+1}-\xi_i)M_{\vp(\xi_i)}\left( \xi_i , \xi_{i+1} \right)+
       \sum_{i={i'}}^{m(\uxi)}(\xi_{i+1}-\xi_i)M_{\vp(\xi_{i+1})}\left( \xi_i , \xi_{i+1}  \right) .
        \end{aligned}
        \end{multline}
        
        From $\uxi\in S_k^*(r, s)$ and our assumption~$\min \{\vp(r), \vp(s) \} \le k - 1$, we have
    \begin{equation*}
    	\label{eq:14}
    	\vp(\xi_{i'})
    	\leq
    	\max \{ k - 1, \vp(r), \vp(s) \}
    \end{equation*}
    and $$	\max_{i \in \{0, \ldots, m(\uxi) + 1 \} \setminus \{ i' \}} \left\{ \vp(\xi_i) \right\} \le k - 1.$$
    
    Together with \eqref{n6}, they imply
    \begin{equation*}
    \begin{multlined}
    	\v\left(\prod_{i=0}^{m(\uxi)}\phi_{\min\{\vp(\xi_i),\vp(\xi_{i+1})\}}(\xi_i,\xi_{i+1})\right) 
    	\geq \min\limits_{\ell \in \{0, \ldots, k - 1 \}} \left\{M_\ell\right\} (s-r)+p^{\vp(r)+\tau}-p^{\vp(\xi_{i'})+\tau}\\\geq  \min_{\ell \in \{0, \ldots, k - 1 \}} \left\{M_\ell\right\}(s-r) +p^{\vp(r)+\tau} - p^{\max\{k-1, \vp(r), \vp(s)\}+\tau}.\qedhere
    \end{multlined}
\end{equation*}\end{proof}

\subsection{Proof of Proposition~\ref{p:congruence-property}}
\label{sec:proof-prop-refp:c}

  Without loss of generality, we assume~$m>0$.
  Put
  $M:=\min_{\ell \in \{0, \ldots, k - 1 \}} \left\{M_\ell\right\}.$
	
	For every pair of integers~$(\beta,\beta')$ such that $r\leq\beta<\beta'\leq s$, put
	$$\Psi(\beta,\beta'):=\begin{cases}
	\Phi(u\beta,u\beta') & \text{if } \beta' < s;\\
	\Phi(u\beta,us)\frac{1-\lambda^{us}}{(1-\lambda^{p^{k+\tau}})\lambda^{us-1}} & \text{if } \beta' = s, \\
	\end{cases}$$
	and
        $$ \epsilon(\beta,\beta')
        :=
        \begin{cases}
          \Phi(u\beta + um, u\beta' + um) - \Psi(u\beta,u\beta') & \text{if } \beta' < s;
          \\
          \Phi(u\beta + um, us + um) \frac{1-\lambda^{us+ um}}{(1-\lambda^{p^{k+\tau}})\lambda^{us +u m - 1}} - \Psi(\beta,s) & \text{if } \beta' = s.
        \end{cases}$$

	Form the definitions of~$\psi_{k}(r + m, s + m)$ and~$\psi_{k}(r,s)$, we have
	\begin{multline}\label{substract}
	\psi_{k}(r + m, s + m) - \psi_{k}(r,s)
	\\
        \begin{aligned}
        & = \sum_{\ubeta\in S_{k}(r, s)}
	\prod_{j=0}^{m(\ubeta)}\left(\Psi(\beta_{j},\beta_{j+1})+\epsilon(\beta_{j},\beta_{j+1})\right)-\sum_{\ubeta\in S_{k}(r, s)}\prod_{j=0}^{m(\ubeta)}\Psi(\beta_{j},\beta_{j+1}).
        \end{aligned}
      \end{multline}
  We write it as a polynomial of variables in the set $\{\epsilon(\beta,\beta')\;|\;r\leq \beta< \beta'\leq s \}$. We know that each monomial in \eqref{substract} is of the form $A\prod_{i=1}^{t}\epsilon(\eta_i,\xi_i)$, where $1\leq t\leq s-r$, $A\in \calK$ and $\{\eta_i\}_{i=1}^t$, $\{\xi_i\}_{i=1}^t$ satisfy that
  \begin{itemize}
  \item	$\eta_i<\xi_i$ for every $1\leq i\leq t;$
  \item
  $	\xi_i\leq \eta_{i+1}$  for every $1\leq i\leq t-1;$
  	\item
  $	p^{k}\nmid \eta_i $ except when $\eta_1= r;$
  	\item
  	$	p^{k}\nmid \xi_i$ except when $ \xi_t= s;$
  		\item
  	it does not contain	$\epsilon(r,s).$
  \end{itemize}

Therefore, to prove this proposition, it is enough to show that for an arbitrary monomial $A\prod_{i=1}^{t}\epsilon(\eta_i,\xi_i)$ we have \begin{equation}\label{12c}
	\v\left(A\prod_{i=1}^{t}\epsilon(\eta_i,\xi_i)\right)\geq M(s-r)-p^{k+\tau-1}.
\end{equation}

To prove this, we first write $A$  explicitly.
Putting $J:=\{1\leq i\leq t-1\;|\; \xi_i<\eta_{i+1}\}$,  we have the following cases:
\begin{enumerate}
	\item If $\eta_1=r$ and $\xi_t=s$, then $A=\prod_{j\in J}\phi_k(\xi_j, \eta_{j+1})$.
	\item If $\eta_1>r$ and $\xi_t=s$, then $A=\phi_k(r,\eta_1)\prod_{j\in J}\phi_k(\xi_j, \eta_{j+1})$.
		\item If $\eta_1=r$ and $\xi_t<s$, then $A=\psi_k(\xi_t,s)\prod_{j\in J}\phi_k(\xi_j, \eta_{j+1})$.
			\item If $\eta_1>r$ and $\xi_t<s$, then $A=\phi_k(r,\eta_1)\psi_k(\xi_t,s)\prod_{j\in J}\phi_k(\xi_j, \eta_{j+1})$.
\end{enumerate}

Note that when $\eta_1>r$, we have $\vp(\eta_1)\leq k-1$. Combined with  Lemma~\ref{key lemma 1}, this implies
    \begin{equation}
	\label{21b}
	\v \left( \phi_k(r, \eta_1) \right)
	\ge
	M (\eta_1 - r).      
\end{equation}

 When $\xi_t=s$, note that 
 \begin{equation*}
 	\epsilon(\eta_t, s) 
 	= \frac{1 - \lambda^{p^\tau}}{1 - \lambda^{p^{k+\tau}}} \left(\psi_0(\xi_t + m, s + m) - \psi_0(\xi_t, s) \right).
 \end{equation*}
 Hence, by Lemma~\ref{level0}(1) we have
 \begin{multline}\label{eq:aa1}
 	\v \left( \epsilon(\eta_t, s) \right)
 	= \v \left( \frac{1 - \lambda^{p^\tau}}{1 - \lambda^{p^{k+\tau}}} \left(\psi_0(\xi_t + m, s + m) - \psi_0(\xi_t, s) \right) \right)
 	\\ \ge
 	M(s - \eta_t) + p^{\vp(m)+2\tau} - p^{k+\tau}
 	\ge
 	M(s - \eta_t).
 \end{multline}

When $\xi_t<s$, by Lemma~\ref{key lemma 1} we have
\begin{equation}\label{20b}
	\v(\psi_k(\xi_t, s))\geq M(s-\xi_t)+p^{\vp(\xi_t)+\tau}-p^{k+\tau};
\end{equation}
and by Lemma~\ref{level0}(2) we have    \begin{multline}
	\label{eq:11}
	\v \left( \epsilon(\eta_i, \xi_i) \right)
	\ge
	M (\xi_i- \eta_i)+p^{\vp(m)+\tau}+p^{\tau}-p^{\vp(\xi_i)+\tau}-p^{\vp(\xi_i+m)+\tau}
	\\
	\geq 	M (\xi_i - \eta_i)+p^{k+\tau}-2p^{\vp(\xi_i)+\tau}. 
\end{multline}
Therefore, we have \begin{multline*}
\v(\epsilon(\eta_t, \xi_t) \psi_k(\xi_t, s))\\
\geq 	M (\xi_t - \eta_t)-2p^{\vp(\xi_t)+\tau}+p^{k+\tau}+M(s-\xi_t)+p^{\vp(\xi_t)+\tau}-p^{k+\tau}\geq M(s-\eta_t)-p^{k+\tau-1}.
\end{multline*}

Combined with \eqref{21b} and \eqref{eq:aa1}, this inequality shows that to prove \eqref{12c}, it is enough to show  \begin{equation*}\label{21a}
		\v\left(\prod_{j\in J}\phi_k(\xi_j, \eta_{j+1})\prod_{i=1}^{t-1}\epsilon(\eta_i,\xi_i)\right)
	\geq  (\eta_t-\eta_1)M.
\end{equation*}
By Lemma~\ref{key lemma 1}, for any $j\in J$ we have
\begin{equation*}
	\v(\phi_k(\xi_j, \eta_{j+1}))\geq M(\eta_{j+1}-\xi_j)-p^{k+\tau-1}+p^{\vp(\xi_j)+\tau} .
\end{equation*}
Combined with \eqref{eq:11}, this inequality implies
 \begin{align*}
	&	\v\left(\prod_{j\in J}\phi_k(\xi_j, \eta_{j+1})\prod_{i=1}^{t-1}\epsilon(\eta_i,\xi_i)\right)
\\	\geq &
	\sum_{j\in J}\left(M(\eta_{j+1}-\xi_j)-p^{k+\tau-1}+p^{\vp(\xi_j)+\tau}\right) +\sum_{i=1}^{t-1} \left(M (\xi_i - \eta_i)+p^{k+\tau}-2p^{\vp(\xi_i)+\tau}\right)\\
	=&(\eta_t-\eta_1)M+	\sum_{j\in J}\left(p^{k+\tau}-p^{k+\tau-1}-p^{\vp(\xi_j)+\tau}\right)
	+	\sum_{j\in \{1,\dots,t-1\}\backslash J}\left(p^{k+\tau}-2p^{\vp(\xi_j)+\tau}\right)\\
	\geq&(\eta_t-\eta_1)M. 
\end{align*}
This completes the proof of this proposition.

\section{Proof of Theorems~\ref{example1} and \ref{example}}\label{sec:application}

\subsection{Proof of Theorem~\ref{example1}}

We first introduce two simple but often used lemmas.
\begin{lemma}\label{5.2}
Given any integers $0\leq r<s$ with $p\;|\;(r+1)$ and $p\nmid(s-r)$, we have
	\begin{equation*}
		\Phi(r,s)=0.
	\end{equation*}
\end{lemma}
\begin{proof}
	Let $\ualpha$ be an arbitrary sequence in $I(r,s)$. 	By our assumption $p\nmid(s-r)=\|\ualpha\|$, there exists at least one term in $\ualpha$ not divisible by $p$. Combined with $p\;|\;(r+1)$, this implies
	\begin{equation*}
		\sum_{i=0}^{s-r}\left\lfloor \frac{\alpha_i}{p}\right\rfloor< \frac{|\ualpha|}{p}=\frac{r+1}{p}=\left\lfloor \frac{r+1}{p}\right\rfloor.
	\end{equation*}

	By Lemma~\ref{fl} with $j=1$, we have $$\binom{r+1}{\ualpha}_p=0,$$
	which completes the proof.
\end{proof}

\begin{lemma}\label{5.3}
Suppose $p\nmid u$. Let any integers $0\leq r<s$ with $p\nmid us+1$ be given. If for a sequence $\ubeta\in S_1(r,s)$ there is $ 0\leq j\leq m(\ubeta)$ such that $p\;|\;(u\beta_{j}+1)$, then 
	\begin{equation*}
		\Phi(\ubeta)=0.
	\end{equation*}
\end{lemma}
\begin{proof}
	Since $p\;|\;(u\beta_{j}+1)$ but $p\nmid (us+1)$, there exists the index  $j\leq j'\leq m(\ubeta)$  such that $p\;|\;(u\beta_{j'}+1)$ but $p\nmid (u\beta_{j'+1}+1)$. By Lemma~\ref{5.2}, we have $\Phi(u\beta_{j'},u\beta_{j'+1})=0$, which completes the proof.
\end{proof}

\begin{proof}[Proof of Theorem~\ref{example1}]  
	By Theorem~\ref{Lindahl1}, we just need to prove that  if  $p\nmid(n+1)$ and $a_n\neq 0$, then 
	 $f(z):=\lambda z+a_nz^{n+1}$ is non-linearizable.
	 
By  Criterion~\ref{c:key criterion}, the theorem can be further reduced to showing that in this case $f$ is $1$-dominant. 
	Note first that in this case, we have $u=n$ and $\tau=\vp(n)$; and by Lemma~\ref{l:0-th}, 
	\begin{equation}\label{revise:2}
		M_0=\v(a_n)-p^\tau.
	\end{equation}

		Consider \begin{equation}\label{21w}
	\phi_1(0,p)=\sum_{\ubeta\in S_1(0,p)} \Phi(\ubeta).
		\end{equation}
	Since for any $\ubeta\in S_1(0,p)$ the term $\Phi(\ubeta)$ is of the form $$\frac{*a_n^p}{\prod_{j=0}^{m(\ubeta)} (1-\lambda^{u\beta_{j+1}})},$$ where $*\in \calK$  satisfies either $*=0$ or $\v(*)=0$, by \eqref{eq:s} we have \begin{equation}\label{21r}
		\v\left(\Phi(\ubeta)\right)\geq p\v(a_n)-m(\ubeta)p^\tau-p^{\tau+1},
	\end{equation}
	with equality if and only if $*\neq 0$.
	
	Now we divide our proof into the following two cases.

\noindent\textbf{Case I:} $p\nmid n$.

We first note that since $p\nmid (n+1),$ this case can only happen for $p\geq 3$.  Set 
	$j'$   to be the unique integer in $\{1,\dots,p-1\}$ such that $p\;|\;j' n+1$. From our assumption $p\nmid n+1$, we have $j'\geq 2$ and hence
	\begin{equation}\label{22e}
		p\nmid u(j'-1). 
	\end{equation}
We now prove that for any $\ubeta\in S_1(0,p)$, we have
	\begin{equation}\label{revise:1}
			\v\left(\Phi(\ubeta)\right)\geq p\v(a_n)-(p-2)p^\tau-p^{\tau+1},
	\end{equation}
	with equality if and only if $\ubeta$ is equal to $\ubeta':=(0,\dots,\widehat{j'},\dots, p)$.

	By Lemma~\ref{5.3}, every $\ubeta\in S_1(0,p)$ with $j'\in \ubeta$ must satisfy $\Phi(\ubeta)=0$. Hence, we are left to consider $\ubeta\in S_1(0,p)$ such that $j'\notin \ubeta$. For such $\ubeta\neq \ubeta'$, we have $m(\ubeta)\geq p-3$ and by \eqref{21r} the inequality \eqref{revise:1} is strict. By \eqref{21r} again, to prove that the equality holds in \eqref{revise:1} for $\ubeta'$ is equivalent to show that 	$\Phi(\ubeta')\neq 0$. Consider
\begin{equation*}
	\Phi(\ubeta')=\prod_{j=0}^{j'-2} \Phi(uj,u(j+1)) \Phi(u(j'-1),u(j'+1))\prod_{\ell=j'+1}^{p-1} \Phi(u\ell,u(\ell+1)).
\end{equation*}	
For any $j\in \{1,\dots,p\}\backslash \{j'\}$ we have $$\Phi(uj,u(j+1))=\frac{a_n\lambda^{uj}}{\lambda(1-\lambda^{u(j+1)})}\binom{uj+1}{uj,0\dots,0, 1}_p\neq 0.$$
From \eqref{22e} and $p\nmid u(j'-1)+1$ we have $$\Phi(u(j'-1),u(j'+1))=\frac{a_n^2\lambda^{u(j'-1)-1}}{\lambda(1-\lambda^{u(j'+1)})}\binom{u(j'-1)+1}{u(j'-1)-1,0\dots,0,2}_p\neq 0.$$
This inequality with the one above proves $\Phi(\ubeta')\neq 0$, and hence $$\v(\phi_1(0,p))=\v(\Phi(\ubeta'))=p\v(a_n)-(p-2)p^\tau-p^{\tau+1}.$$

	Combining it with \eqref{eq:2} and \eqref{revise:2}, we have 
	$$M_1(0,p)=\frac{\v(\phi_1(0,p))+p^{\tau+1}-p^{\tau+1}}{p}=M_0-(p-2)p^{\tau-1}\leq M_0-p^{\tau-1},$$
	where the last inequality is from the necessary condition $p\geq 3$ for this case. Hence, $f$ is $1$-dominant in this case.
	
	\noindent\textbf{Case II:}  $p|n$.
	
	 For every $j\in \{0,\dots,p-1\}$ we have  
 $$\Phi(uj,u(j+1))=\frac{a_n\lambda^{uj}}{\lambda(1-\lambda^{u(j+1)})}\binom{uj+1}{uj,0\dots,0,1}_p=\frac{a_n\lambda^{uj}}{\lambda(1-\lambda^{u(j+1)})}\neq 0.$$
Therefore, by \eqref{21r} the sequence  $\ubeta'':=(0,1,\dots,p)$ uniquely maximizes $\v(\Phi(\ubeta''))$ among all $\ubeta\in S_1(0,p)$. Thus,  
\begin{equation*}
	\v(\phi_1(0,p))=\v(\Phi(\ubeta''))=\sum_{j=0}^{p-1} \v\left(\frac{a_n\lambda^{uj}}{\lambda(1-\lambda^{u(j+1)})}\right)=p\v(a_n)-2p^{\tau+1}+p^{\tau}.
\end{equation*}  
Combined with \eqref{revise:2}, this equality implies that
	$$M_1(0,p)=\frac{\v(\phi_1(0,p))+p^{\tau+1}-p^{\tau+1}}{p}=M_0-p^{\tau}+p^{\tau-1}\leq M_0-p^{\tau-1},$$
	and hence that  $f$ is $1$-dominant. 
\end{proof}

%
%

\subsection{Proof of Theorem~\ref{example}}
In the rest of this section, we focus on polynomials of the form $f(z)=\lambda z +a_1z^2+a_2z^3$.
For every pair of integers $0\leq r\leq s$ we put
 $$\tilde I(r,s):=\{(\alpha_0,\alpha_1,\alpha_2)\in \NN^3\;|\; (\alpha_0,\alpha_1,\alpha_2,0,0,\dots)\in I(r,s)\}.$$

Namely, $\tilde I(r,s)$ contains all the solutions $(\alpha_0,\alpha_1,\alpha_2)\in \NN^3$
to the system of equations:
\begin{equation}\label{eq:solution}
\begin{cases}
\alpha_0+\alpha_1+\alpha_2=r+1,\\
\alpha_1+2\alpha_{2}=s-r.
\end{cases}
\end{equation}

Note that 
$$
\Phi(r,s)=\frac{1}{\lambda(1-\lambda^{s})}\sum_{\ualpha\in \tilde I(r,s)} \binom{r+1}{\ualpha}_p\lambda^{\alpha_0} a_1^{\alpha_1} a_{2}^{\alpha_{2}}.
$$

\begin{proof}[Proof of Theorem~\ref{example}] In the view of Criterion~\ref{c:key criterion}, it is sufficient to prove that $f$ is $1$-dominant. 
	
	Note that from our assumption $a_1\neq 0$, we have $u=1$ and $\tau=0$ in this case.
Putting $\mathbf{m}_{0}:=\min \left\{\v\left(a_{1}\right), \frac{\v\left(a_{2}\right)}{2}\right\},$ for every $\left(\alpha_{0}, \alpha_{1}, \alpha_{2}\right)$ in $\tilde{I}(r, s)$ we have
\[
\v\left(\lambda^{\alpha_{0}} a_{1}^{\alpha_{1}} a_{2}^{\alpha_{2}}\right) \geqslant(r-s) \mathbf{m}_{0}.
\]
This implies
\begin{equation}\label{u1}
\v(\Phi(r, s)) \geqslant(s-r) \mathbf{m}_{0}-p^{\vp(s)}.
\end{equation}

The proof is divided into two cases. 

\noindent \textbf{Case I}: $\left|1-\frac{a_{2}}{a_{1}^{2}}\right| \in\left[1, \frac{1}{|1-\lambda|}\right) .$ By \eqref{eq:s}, this is equivalent to
\begin{align}
	&	\v\left(a_{1}^{2}-a_{2}\right) \leqslant 2 \v\left(a_{1}\right)\label{23q};\\
	&\v\left(a_{1}^{2}-a_{2}\right)>	2 \v\left(a_{1}\right)-1\label{23w}.
\end{align}

Now we prove \begin{equation}\label{23e}
	 \v\left(a_{1}^{2}-a_{2}\right)=2 \mathbf{m}_{0}.
\end{equation}
By the strong triangular inequality, this is trivial for $2\v(a_1)\neq \v(a_2)$. If $2\v(a_1)= \v(a_2)$, 
by  $$\v(a_1^2-a_2)\geq \min\{2\v(a_1), \v(a_2)\}=2\v(a_1)$$ and \eqref{23q}, we obtain \eqref{23e} again.

Combining \eqref{23w} with \eqref{23e}, we have 
\begin{equation}\label{u2'}
	 \frac{\v\left(a_{2}\right)-1}{2}\geq  \frac{2\mathbf{m}_0-1}{2}=\frac{\v(a_1^2-a_2)-1}{2}>\v\left(a_{1}\right)-1.
\end{equation}
Hence, by Lemma~\ref{l:0-th} this implies 
\begin{equation}\label{u2}
M_{0}=\v\left(a_{1}\right)-1\geq \mathbf{m}_{0}-1.
\end{equation}

%
%
%
%
%
%
%
%
%

To prove that $f$ is $1$-dominant, we estimate $M_{1}(0,p) .$ Let $r$ and $s$ in $\{0, \ldots, p-1\}$ be such that $s>r .$ If $s-r \geqslant 2,$ then we have
\[
\frac{\v(\Phi(r, s))}{s-r} \geqslant \mathbf{m}_{0}-\frac{1}{s-r}>M_{0},
\]
where the first equality follows from \eqref{u1} and the second from \eqref{23w}, \eqref{23e} and \eqref{u2}.
On the other hand, in the case $r=s-1$ we have
\[
\v(\Phi(s-1, s))=\v\left(a_{1}\right)-1=M_{0}.
\]
It follows that for every $\ubeta \in S_{1}(0, p),$ we have
\begin{equation}\label{u3}
\v(\Phi(\underline{\beta})) \geqslant\left(p-\beta_{m(\underline{\beta})}\right) M_{0}+\v\left(\Phi\left(\beta_{m(\underline{\beta})}, p\right)\right)
\end{equation}
with equality if and only if $\underline{\beta}=\left(0,1, \ldots, \beta_{m(\underline{\beta})}, p\right).$ On the other hand,
\begin{equation}\label{revise:3}
\Phi(p, p-1)=0 \quad\text {and}\quad \Phi(p-2, p)=\frac{1}{\mu^{p}}\left(\lambda^{p-4}a_{1}^{2}-\lambda^{p-3}a_{2}\right).
\end{equation}
By $$\v\left(\left(\lambda^{p-4}a_{1}^{2}-\lambda^{p-3}a_{2}\right)-(a_1^2-a_2)\right)\geq \max\left\{\v\left((\lambda^{p-4}-1)a_{1}^{2}\right), \v\left((1-\lambda^{p-3})a_{2}\right)\right\}>2\mathbf{m}_0$$
and \eqref{u2'}, we have 
\begin{equation}\label{21t}
	\v(\Phi(p-2,p))=2\mathbf{m}_0-p.
\end{equation}
For every $r$ in $\{0, \ldots, p-3\}$  by \eqref{u1}, \eqref{23e} and \eqref{u2} we have
\[
\frac{\v(\Phi(r, p))}{p-r} \geqslant \mathbf{m}_{0}-\frac{p}{p-r}>\mathbf{m}_{0}-\frac{p}{2}=\frac{\v(\Phi(p-2, p))}{2}.
\]
Combined with \eqref{u3}, \eqref{revise:3} and \eqref{21t}, this implies
\[
\v(\Phi(\underline{\beta})) \geqslant(p-2) M_{0}+2 \mathbf{m}_{0}-p
\]
with equality if and only if $\ubeta=(0,1, \ldots, p-2, p) .$ Together with \eqref{new6}, \eqref{u2} and our hypothesis $p \geqslant 5,$ this implies
\[
M_{1}(0,p)=\frac{\v\left(\phi_{1}(0, p)\right)}{p}=\frac{(p-2) M_{0}+2 \mathbf{m}_{0}-p}{p}\leq  M_{0}+\frac{2-p}{p}\leq  M_{0}-p^{-1}.
\]
This proves that $f$ is 1 -dominant.

\noindent \textbf{Case II}: $\left|1-\frac{a_{2}}{a_{1}^{2}}\right|>\frac{1}{|1-\lambda|} .$ This is equivalent to
\begin{equation}\label{revise:4}
\v\left(a_{2}\right)<2 \v\left(a_{1}\right)-1,
\end{equation}
and by Lemma~\ref{l:0-th} this implies
\begin{equation}\label{u4}
\mathbf{m}_{0}=\frac{\v\left(a_{2}\right)}{2}<\v\left(a_{1}\right) \textrm{\ and\ } M_{0}=\frac{\v\left(a_{2}\right)-1}{2}=\mathbf{m}_{0}-\frac{1}{2}.
\end{equation}

To prove that $f$ is $1$-dominant, we estimate $M_{1}(0,2) .$ Note that by \eqref{u1} and \eqref{u4}
for every $r$ and $s$ in $\{0, \ldots, 2 p\}$ such that $s-r \geqslant 2,$ we have
\begin{equation}\label{u5}
\v(\Phi(r, s)) \geqslant(s-r) M_{0}+1-p^{\vp(s)}
\end{equation}
with strict inequality if $s-r \geqslant 3 .$ Consider on the other hand $r$ and $s$ in $\{0, \ldots, 2 p\}$ such that $s>r$ and such that $s-r$ is odd. Then for every $\left(\alpha_{0}, \alpha_{1}, \alpha_{2}\right)$ in $\tilde{I}(r, s),$ we have
\begin{equation}\label{u6}
\v\left(\lambda^{\alpha_{0}} a_{1}^{\alpha_{1}} a_{2}^{\alpha_{2}}\right) \geqslant \v\left(a_{1}\right)+(s-r-1) \mathbf{m}_{0}.
\end{equation}
Together with \eqref{revise:4} and \eqref{u4}, this implies
\begin{multline}\label{u7}
\v(\Phi(r, s)) \geqslant \v\left(a_{1}\right)+(s-r-1) \mathbf{m}_{0}-p^{\mathrm{val}_{p}(s)}>(s-r) \mathbf{m}_{0}+\frac{1}{2}-p^{\mathrm{val}_{p}(s)}
\\
\geqslant(s-r) M_{0}+1-p^{\vp(s)}.
\end{multline}
Combined with \eqref{u5} this implies that for all $r$ and $s$ in $\{0, \ldots, 2 p\}$ such that $s>r$ and every $\underline{\beta}$ in $S_{1}(r, s),$ we have
\begin{equation}\label{u8}
\v(\Phi(\underline{\beta})) \geqslant(s-r) M_{0}+1-p^{\vp(s)},
\end{equation}
with strict inequality if for some $j$ in $\{0, \ldots, m(\beta)\}$ we have $\beta_{j+1}-\beta_{j} \neq 2$.

 To estimate $M_{2}(0,2p),$ let $\ubeta$ in $S_{1}(0,2 p)$ be given and let $j_{0}$ be the index in $\{0, \ldots, m(\underline{\beta})\}$ satisfying $\beta_{j_{0}}<p<\beta_{j_{0}+1} .$ Applying \eqref{u8} with $r=0$ and $s=\beta_{j_0}$ if $\beta_{j_{0}}>0,$ and with $r=\beta_{j_{0}+1}$ and $s=2 p$ if $\beta_{j_{0}+1}<2 p,$ we obtain
\begin{multline}\label{u9}
\v(\Phi(\underline{\beta}))\\
\geqslant \v\left(\Phi\left(\beta_{j_{0}}, \beta_{j_{0}+1}\right)\right)+
\begin{cases}
\left(2 p-\left(\beta_{j_{0}+1}-\beta_{j_{0}}\right)\right) M_{0}+1-p & \text { if } \beta_{j_{0}+1}<2 p, \\
\beta_{j_{0}} M_{0} & \text { if } \beta_{j_{0}+1}=2 p,
\end{cases}
\end{multline}
with strict inequality if for some $j$ in $\{0, \ldots, m(\underline{\beta})\}$ different from $j_{0}$ we have $\beta_{j+1}-\beta_{j} \neq 2$. 

If $\beta_{j_{0}+1}-\beta_{j_{0}}=2,$ then $\beta_{j_{0}}=p-1, \beta_{j_{0}+1}=p+1$ and therefore $\Phi\left(\beta_{j_{0}}, \beta_{j_{0}+1}\right)=0$.

If $\beta_{j_{0}+1}-\beta_{j_{0}}=3,$ then by \eqref{u4} and the second inequality in \eqref{u7} we have
\[
\v\left(\Phi\left(\beta_{j_{0}}, \beta_{j_{0}+1}\right)\right)>3 M_{0}+2-p^{\vp\left(\beta_{j_{0}+1}\right)}.
\]
Finally, if $\beta_{j_{0}+1}-\beta_{j_{0}} \geqslant 4,$ then by \eqref{u1} and \eqref{u4} we have
\[
\v\left(\Phi\left(\beta_{j_{0}}, \beta_{j_{0}+1}\right)\right) \geqslant\left(\beta_{j_{0}+1}-\beta_{j_{0}}\right) M_{0}+2-p^{\vp\left(\beta_{j_{0}+1}\right)},
\]
with strictly inequality if $\beta_{j_{0}+1}-\beta_{j_{0}} \geqslant 5 .$ In all the cases, this last inequality is strict  if $\beta_{j_{0}+1}-\beta_{j_{0}} \neq 4 .$ Together with \eqref{u9}, this implies
\begin{equation}\label{u10}
\v(\Phi(\underline{\beta})) \geqslant 2 p M_{0}+2-p,
\end{equation}
with strict inequality if $\beta_{j_{0}+1}-\beta_{j_{0}} \neq 4$ or if for some $j$ in $\{0, \ldots, m(\beta)\}$ different from $j_{0}$ we have $\beta_{j+1}-\beta_{j} \neq 2 .$ Equivalently, the inequality above is strict unless  $\ubeta$ is equal to the increasing sequence $\underline{\beta'}$ of all those even integers in $\{0, \ldots, 2 p\}$ different from $p-1 .$ We now
verify that for $\underline{\beta}=\underline{\beta'}$ the inequality above holds with equality. Note first that for every $s$ in $\{2, \ldots, 2 p\}$ we have
\[
\Phi(s-2, s)=\frac{1}{\lambda\left(1-\lambda^{s}\right)}\left(\frac{(s-1)(s-2)}{2} \lambda^{s-3}a_{1}^{2}+(s-1)\lambda^{s-2} a_{2}\right).
\]
Thus, if in addition $s \neq p+1,$ then by \eqref{u4}
\begin{equation}\label{u11}
\v(\Phi(s-2, s))=\v\left(a_{2}\right)-p^{\vp(s)}=2 M_{0}+1-p^{\vp(s)}.
\end{equation}
On the other hand, a direct computation shows that
\[
\Phi(p-3, p+1)=\frac{1}{\lambda\left(1-\lambda^{p+1}\right)}\left(5 \lambda^{p-6}a_{1}^{4}-12 \lambda^{p-5}a_{1}^{2} a_{2}+3 \lambda^{p-4}a_{2}^{2}\right),
\]
so by \eqref{u4} and our hypothesis $p \geqslant 5$ we have
\[
\v(\Phi(p-3, p+1))=2 \v\left(a_{2}\right)-1=4 M_{0}+1.
\]
Combined with \eqref{u11} this implies that \eqref{u10} holds with equality for $\ubeta=\ubeta' .$ Since the inequality is strict if $\ubeta \neq \ubeta',$ by \eqref{eq:4} and our hypothesis $p \geqslant 5$ this proves
\[
2 p M_{2}(0,2p)=\v\left(\phi_{1}(0,2 p)\right)=2 p M_{0}+2-p \leqslant 2 p M_{0}-2.
\]
This proves that $f$ is $1$-dominant and completes the proof of the theorem.
\end{proof}

\section{Proof of Theorem~\ref{Thm:p-1}}\label{sec:application1}

Note that in Theorem~\ref{Thm:p-1}, if $a_1=0$ then the polynomial $\lambda z+a_1z^2+a_{p-1}z^p$ degenerates to the case in Theorem~\ref{Lindahl1}, and is linearizable. Therefore, to prove Theorem~\ref{Thm:p-1} it is enough to show that  if $a_1\neq 0$, then $\lambda z+a_1z^2+a_{p-1}z^p$ is non-linearizable. For this sake, 
we  assume $a_1\neq 0$ in the whole section;  as a consequence we have $ u=1$ and $\tau=0$. 

\begin{notation}
	In this section, given any integers $0\leq r<s$ we treat $$I'(r,s):=\{(\alpha_0,\alpha_1,\alpha_{p-1})\in \NN^3\;|\; \alpha_0+\alpha_1+\alpha_{p-1}=r+1 \textrm{~and~} \alpha_1+(p-1)\alpha_{p-1}=s-r\}$$
	as a subset of $I(r,s)$ by mapping $(\alpha_0,\alpha_1,\alpha_{p-1})$ to $(\alpha_0,\alpha_1,0,\dots, 0, \alpha_{p-1},0,\dots)$.
\end{notation}

\begin{notation}\label{revise:5}
	\noindent
	\begin{enumerate}
		\item 	Let $T:=\v(a_1)-\frac{\v(a_{p-1})-1}{p-1}$. 
		\item Given any $p^k$-divisible integers $0\leq r<s$ with $(p-1)\;|\;s-r$, we put
		$$\ualpha(r,s):=\left(r+1-\frac{s-r}{p-1},0,\frac{s-r}{p-1}\right)\in I'(r,s).$$
	\end{enumerate}
\end{notation}

\begin{lemma}\label{alpha}
	Given any $k\geq 0$ and any $p^k$-divisible integer $r\geq 0$, we have 
	$$\binom{r+1}{\ualpha\left(r,r+p^k(p-1)\right)}_p=\begin{cases}
		r+1	 &\textrm{for~}k=0,\\
		r/p^k  &\textrm{for~} k\geq 1.
	\end{cases}$$
\end{lemma}
\begin{proof}
	The case for $k=0$ is trivial. For $k\geq 1$, we have $$\binom{r+1}{r+1-p^k,0,p^k}_p=\frac{(r+1)\prod\limits_{i=r-p^k+2}^{r-1}i}{(p^k-1)!}\times \frac{r}{p^k}. $$
	From our assumptions $k\geq 1$ and $p^k|r$, we have $$(r+1)\prod\limits_{i=r-p^k+2}^{r-1}i\equiv (p^k-1)!\pmod p,$$ and hence complete the proof.
\end{proof}


\begin{lemma}\label{key lemma}
	Assume $T>1.$ 
	\begin{enumerate}
		
		\item Let any $k\geq 0$ such that $p^{k-1}<T$ and $p^k$-divisible integers $0\leq r<s$ be given. If $(p-1)\nmid s-r$, we have 
		$$\v(\phi_k(r, s))\geq M_0(s-r)+T-p^{\vp(s)},$$
		with equality if and only if $\begin{cases}
			s-r=1 \textrm{~and~} p\nmid r+1& \textrm{for~} k=0,\\
			s-r=p^k& \textrm{for~} k\geq 1.
		\end{cases}$
		\item Given any $k\geq 0$ such that $p^{k-1}<T$ and $p^k$-divisible integers $0\leq r<s$ with $(p-1)\;|\;s-r$, we have
		$$\v\left(\phi_k(r, s)-\frac{1}{\lambda(1-\lambda^s)}\binom{r+1}{\ualpha(r,s)}_p\ua^{\ualpha(r,s)}\right)> M_0(s-r)+T-p^{\vp(s)}.$$
		
		\item Given any $k\geq 0$ such that $p^k<T$ and $p^k$-divisible integers $0\leq r<s$, we have $$\v(\phi_k(r, s))\geq M_0(s-r)+p^k-p^{\vp(s)},$$ 
		with equality if and only if  $s-r=(p-1)p^k$ and 
		$
		\begin{cases}
			p\nmid r+1 &\textrm{for~}k=0 ,\\
			p^{k+1}\nmid r  &\textrm{for~} k\geq 1.
		\end{cases}
		$
		%
	\end{enumerate}
\end{lemma}
\begin{proof}

	Note first that by Lemma~\ref{l:0-th}, our assumption $T>1$ implies $M_0=\frac{\v(a_{p-1})-1}{p-1}$.
Hence, we have
	\begin{equation}\label{11b}
		\v(a_1)=M_0+T\quad \textrm{and}\quad \v(a_{p-1})=(p-1)M_0+1.
	\end{equation}

		We prove this lemma by taking induction across (1)--(3).
	
	\noindent	\textbf{Step I}: Proving (1) and (2) for $k=0$.
	
	For $k=0$, given any $0\leq r<s$ we have
	$$\phi_0(r,s)=\Phi(r,s)=\frac{1}{\lambda(1-\lambda^s)}\sum_{\ualpha\in I'(r,s)}\binom{r+1}{\ualpha}_p\ua^{\ualpha}.$$
	Note that for any given $\ualpha\in I'(r,s)$, we have $\alpha_1+(p-1)\alpha_{p-1}=s-r$ and
	\begin{equation}\label{10a}
		\v\left(\frac{1}{\lambda(1-\lambda^s)}\binom{r+1}{\ualpha}_p\ua^{\ualpha}	\right)
		\geq \alpha_1\v(a_1)+\alpha_{p-1}\v(a_{p-1})-p^{\vp(s)},
	\end{equation}
	with equality if and only if $\binom{r+1}{\ualpha}_p\neq 0$.
	
	Plugging \eqref{11b} in \eqref{10a}, we obtain that 
	\begin{align*}
		\textrm{the right side of \eqref{10a}}=& \alpha_1(M_0+T)+\alpha_{p-1}((p-1)M_0+1)-p^{\vp(s)} \\
		=& 
		\alpha_1T+(s-r)M_0+\alpha_{p-1}-p^{\vp(s)}.
	\end{align*}
	Therefore, for $\ualpha\in I'(r,s)$ with $\alpha_1\geq 1$ we have
	\begin{equation}\label{10b}
		\v\left(\frac{1}{\lambda(1-\lambda^s)}\binom{r+1}{\ualpha}_p\ua^{\ualpha}	\right)
		\geq 	(s-r)M_0+T-p^{\vp(s)},
	\end{equation}
	with equality if and only if $\alpha_1=1$, $\alpha_{p-1}=0$ and $\binom{r+1}{\ualpha}_p\neq 0$. 
	Note that in this case $$s-r=\alpha_1+(p-1)\alpha_{p-1}=1$$ and 
	$\binom{r+1}{\ualpha}_p\neq 0$ is equivalent to $p\nmid r+1$.
	
	If $p-1\nmid s-r$,   then $\alpha_1\geq 1$ for every $\ualpha\in I'(r,s)$. Combined with \eqref{10b}, this property proves (1) for $k=0$.
	
	If $p-1\;|\;s-r$, then $\ualpha(r,s)$ is the unique sequence in $I'(r,s)$ with $\alpha(r,s)_1=0$.
	Combined with \eqref{10b}, this completes the proof of (2) for $k=0$.
	
	%
	
	%
	%
	%
	

	\noindent	\textbf{Step II}: 
	Given any $k\geq 0$ with $p^k< T$, assuming that (1) and (2) hold for $k$, we prove that so does (3).

	Clearly, the common condition in (1) and (2) for $k$, i.e. $p^{k-1}<T$, is satisfied.  
	Therefore, for any $0\leq r<s$ such that $p^{k}\;|\;r$ and $p^{k}\;|\;s$, if $p-1\nmid s-r$ or $p-1\;|\; s-r$ and  $\binom{r+1}{\ualpha(r,s)}_p= 0$, by our induction on (1) and (2), we have $$\v\left(\phi_k(r, s)\right)>M_0(s-r)+T-p^{\vp(s)}>M_0(s-r)+p^k-p^{\vp(s)}.$$

	Thus, the only case left to study is that $p-1\;|\;s-r$ and $\binom{r+1}{\ualpha(r,s)}_p\neq 0$, in which
	$$\v\left(\frac{1}{\lambda(1-\lambda^s)}\binom{r+1}{\ualpha(r,s)}_p\ua^{\ualpha(r,s)}\right)=M_0(s-r)+\frac{s-r}{p-1}-p^{\vp(s)}\geq M_0(s-r)+p^k-p^{\vp(s)},$$
		with equality if and only $s-r=(p-1)p^k$. 
		
	Combined with our induction hypothesis on (2), this implies
	$$\v\left(\phi_k(r, s)\right)\geq M_0(s-r)+p^k-p^{\vp(s)},$$
	with equality if and only $s-r=(p-1)p^k$. 
	
	By Lemma~\ref{alpha},  for $p^k$-divisible integers $0\leq r<s$ with $s-r=(p-1)p^k$,  the condition $\binom{r+1}{\ualpha(r,s)}_p\neq 0$ is equivalent to
	$\begin{cases}
		p\nmid r+1 &\textrm{for~}k=0\\
		p^{k+1}\nmid r  &\textrm{for~} k\geq 1
	\end{cases}$.
	This completes the proof.

	\noindent	\textbf{Step III}: 
		Given any $k\geq 0$ with $p^k< T$, assuming that (1)--(3) hold for $k$, we prove (1) for $k+1$.
		
	We first note that the conditions in (1)--(3) for $k$,  i.e. $p^{k-1}<T$ and $p^k<T$, are satisfied. 
	Let any $p^{k+1}$-divisible integers $0\leq r<s$ be given. By Lemma~\ref{l:coefficient-formula1},  we have
	\begin{equation}\label{10e}
		\phi_{k+1}(r,s)=\sum_{\ubeta\in S_1(r/p^k,s/p^k)}\prod_{j=0}^{m(\ubeta)}\phi_k(\beta_jp^k,\beta_{j+1}p^k).
	\end{equation}

Thus, 	to complete this inductive step, it is enough to show that if $p-1\nmid s-r$, then for every $\ubeta\in S_1(r/p^k,s/p^k)$,
 \begin{equation}\label{10d}
		\v\left(\prod_{j=0}^{m(\ubeta)}\phi_k(\beta_{j}p^k,\beta_{j+1}p^k)\right)
		>M_0(s-r)+T-p^{\vp(s)}.
	\end{equation}

	We prove \eqref{10d} in three cases, and note first that $p\nmid \beta_{j}$ for every $1\leq j\leq m(\ubeta)$.
	
	\noindent\textbf{Case I:} When $p-1\nmid( \beta_1p^k-r)$ and $\beta_1p^k-r\geq 2p^k$,  by induction on (1) for $k$, we have
	\begin{equation}\label{10c}
		\v(\phi_k(r, \beta_1p^k))>M_0(\beta_1p^k-r)+T-p^k.
	\end{equation}
	
	\noindent\textbf{Case II:} When $p-1\;|\;( \beta_1p^k-r)$, from our assumption $p-1\nmid s-r$, we have $\beta_1p^k\neq s$ and hence $p\nmid\beta_1$.
	
	If $k=0$ and  $p\;|\;(\beta_1+1)$, then by Lemma~\ref{5.3}, we have $$	\prod_{j=0}^{m(\ubeta)}\phi_0(\beta_{j},\beta_{j+1})=\Phi(\ubeta)=0.$$
	
	If $k\geq 1$, we have $\frac{\beta_1p^k-r}{p-1}>1$. 
	Combined with $p\nmid \beta_1$, this inequality implies
	$$\left\lfloor\frac{\beta_1p^k-r}{(p-1)p^{k+1}}\right\rfloor+\left\lfloor\frac{ r+1-\frac{\beta_1 p^k-r}{p-1}}{p^{k+1}}\right\rfloor=r/p^{k+1}-1.$$
	
	Considering
	$\left\lfloor\frac{r+1}{p^{k+1}}\right\rfloor=r/p^{k+1}$,
	by Lemma~\ref{fl}
	we have
	$$\binom{r+1}{\ualpha(r,\beta_1p^k)}_p=\binom{r+1}{r+1-\frac{\beta_1 p^k-r}{p-1},0,\frac{\beta_1p^k-r}{p-1}}_p=0.$$
	Therefore, from induction hypothesis on (2) for $k$, we obtain \eqref{10c} again.
	
	On the other hand, from induction hypothesis on (3), for every $1\leq j\leq m(\ubeta)$ we have 
	$$\v\left(\phi_k(\beta_jp^k,\beta_{j+1}p^k)\right)\geq M_0(\beta_{j+1}-\beta_j)p^k+p^k-p^{k+\vp(\beta_{j+1})}.$$
	
	Together with \eqref{10c}, this shows 
	\begin{align*}
		&	\v\left(\prod_{j=0}^{m(\ubeta)}\phi_k(\beta_{j}p^k,\beta_{j+1}p^k)\right)
		\\	>&
		M_0(\beta_1p^k-r)+T-p^k+\sum_{j=1}^{m(\ubeta)-1}M_0(\beta_{j+1}-\beta_{j})p^k+ M_0(s-\beta_{m(\ubeta)}p^k)+p^{k}-p^{\vp(s)}
		\\
		=&M_0(s-r)+T-p^{\vp(s)},
	\end{align*}
	and hence proves \eqref{10d} for Cases I and II.
	
	\noindent\textbf{Case III:} Now we study the case left that $\beta_1p^k-r=p^k$. First, we prove 
	\begin{equation}\label{18a}
		\v(\phi_k(r, \beta_1p^k))=	\v(\phi_k(r, r+p^k))= M_0p^k+T-p^k.
	\end{equation}
	By induction hypothesis on (1) for $k$, it is enough to show that $\beta_1p^k-r=p^k$ implies the equality condition in (1). 
	Clearly, we just need to prove $p\nmid r+1$ for $k=0$. It follows directly from our assumption $p^{k+1}\;|\;r$.

	By induction hypothesis on (3) on $k$,   for every $j\in \{1,\dots,m(\ubeta)\},$
	\begin{equation}\label{18b}
		\v\left(\phi_k(\beta_{j}p^k,\beta_{j+1}p^k)\right)\geq  M_0(\beta_{j+1}-\beta_{j})p^k+p^k-p^{k+\vp(\beta_{j+1})},
	\end{equation}
	with equality if and only if $\beta_{j+1}-\beta_j=p-1$ and 	
	\begin{equation}\label{11c}
		\begin{cases}
			p\nmid \beta_j+1 &\textrm{for~}k=0, \\
			p\nmid \beta_j  &\textrm{for~} k\geq 1.
		\end{cases}
	\end{equation}

	Combining \eqref{18a} and \eqref{18b}, we have
	\begin{equation}\label{10g}
		\v\left(\prod_{j=0}^{m(\ubeta)}\phi_k(\beta_{j}p^k,\beta_{j+1}p^k)\right)
		\geq M_0(s-r)+T-p^{\vp(s)}.
	\end{equation}

	We next prove that if $s-r>p^{k+1}$, then the equality in \eqref{10g} holds for none of $\ubeta\in S_1(r/p^k,s/p^k)$.
	Suppose, on contrary, that there is $\ubeta'$ with $\beta_1'p^k=r+p^k$ that holds this equality. Then we must have $\beta'_{j+1}-\beta'_j=p-1$ and $\beta_j$ satisfies \eqref{11c} for every $1\leq j\leq m(\ubeta')$. 
	Combined with $s-r>p^{k+1}$ and $p^{k+1}\;|\;s$, this implies $m(\ubeta)\geq p$. We write first several terms of $\ubeta$ explicitly: $$\beta_2'=r/p^k+p, \ \beta_3'=r/p^k+2p-1 \textrm{~and~} \beta_4'=r/p^k+3p-2. $$
	
	For $k=0$, since $\beta_3'=r/p^k+2p-1$,  the term $\beta_3$ fails \eqref{11c}, a contradiction. 
	For $k\geq 1$, since $\beta_2'p^k=r+p^{k+1}$ is divisible by $p^{k+1}$, the term $\beta_2$ fails \eqref{11c}, a contradiction again. 
	
	Now for $s=r+p^{k+1}$, since $\begin{cases}
		p\nmid r+1+1& \textrm{for~} k=0\\
		p^{k+1}\nmid r+p^k& \textrm{for~} k\geq 1\\
	\end{cases}$, by induction on (3),
	we have
	$$\v\left(\phi_k(r+p^k,r+p^{k+1})\right)=  M_0(p-1)p^k+p^k-p^{\vp(r+p^{k+1})}.$$ Hence,
	the equality in \eqref{10g} holds uniquely for 
	$(r/p^k,r/p^k+1,r/p^k+p)\in S_1(r/p^{k},r/p^k+p)$,
	where $r$ is an arbitrary $p^{k+1}$-divisible integer.

	Combining Cases I--III, we complete the proof of this step.

	%
	%
	%
	%
	%

	\noindent	\textbf{Step IV}: 	Given any $k\geq 0$ with $p^k< T$, assuming that (1)--(3) hold for $k$, we prove (2) for $k+1$.
	
We note again that the conditions in (1)--(3) for $k$,  i.e. $p^{k-1}<T$ and $p^k<T$, are satisfied. 
	Let any $p^{k+1}$-divisible integers $0\leq r<s$ with $p-1\nmid s-r$ be given. Similar to the proof of Step III, we consider \eqref{10e}. By induction on (2) for $k$, we have
	$$\v\left(\phi_k(r, s)-\frac{1}{\lambda(1-\lambda^s)}\binom{r+1}{\ualpha(r,s)}_p\ua^{\ualpha(r,s)}\right)> M_0(s-r)+T-p^{\vp(s)}.$$
	Since $\phi_k(r,s)$ is same to the term in \eqref{10e} corresponding to $(r/p^k,s/p^k)$,
	it is enough to show that every $\ubeta\in S_1(r/p^k,s/p^k)\backslash (r/p^k,s/p^k)$ satisfies \eqref{10d}.
	We also divide the argument into three cases.
	
	\noindent\textbf{Case I:} When $p-1\nmid( \beta_1p^k-r)$ and $\beta_1p^k-r\geq 2p^k$, we prove \eqref{10d} with the same argument to the one of Case I in Step III.

	\noindent\textbf{Case II:} When $p-1\;|\;( \beta_1p^k-r)$, note that $\ubeta\neq (r/p^k, s/p^k)$ implies $p\nmid \beta_1$.  We then prove \eqref{10d} with the same argument to the one of Case II in Step III. .

	\noindent\textbf{Case III:} Now the only case left is $\beta_1p^k-r=p^k$. 
	
	Note that $p-1\nmid ((s-r)/p^k-1)$. Therefore, there exists  $1\leq j'\leq m(\ubeta)$ such that $p-1\nmid \beta_{j'+1}-\beta_{j'}$. By the induction on (1) for $k$, we have
	$$\v\left(\phi_k(\beta_{j'}p^k,\beta_{j'+1}p^k)\right)> M_0(\beta_{j'+1}-\beta_{j'})p^k+T-p^{k+\vp(\beta_{j'+1})}$$
	and
	\begin{equation*}
		\v(\phi_k(r, r+p^k))= M_0p^k+T-p^k.
	\end{equation*}
	By the induction on (3),   for every $j\in \{1,\dots,m(\ubeta)\}\backslash\{j'\},$
	$$\v\left(\phi_k(\beta_{j}p^k,\beta_{j+1}p^k)\right)\geq  M_0(\beta_{j+1}-\beta_{j})p^k+p^k-p^{k+\vp(\beta_{j+1})}.$$
	
	Combining the three formulas above with our assumption $T>p^k$, we prove \eqref{10d} for Case III; and this completes the proof of induction. 
\end{proof}

%
%
%
%

%
%
%

\begin{notation}
	For any $A,B$ in $\calK$, we denote by $A\sim B$ if $\v(A)<\v(A-B)$. 
\end{notation}
Note that $0\not\sim0$. 

\begin{lemma}\label{tri}
	\noindent\begin{enumerate}
		\item For any $A, B$ in $\calK$, if $A\sim B$, then $\v(A)=\v(B)$.
		\item For any $A,B,C$ in $\calK$ with $C\neq 0$ if $A\sim B$, then $AC\sim BC$.
		\item 	The ``$\sim$'' is an equivalence relation. 
	\end{enumerate}
	
	We will not state $C\neq 0$ while using (2) when this condition is clear.  
\end{lemma}
\begin{proof}
	
	\noindent
	(1) From the strong triangular inequality, we have
	$$\v(A)=\v(A-B+B)\geq \min\{\v(A-B), \v(B)\}.$$
	Combined with our assumption $\v(A)<\v(A-B)$, this implies $\v(A)\geq \v(B)$.
	Similarly, we have $$\v(B)=\v(A-(A-B))\geq \min\{\v(A),\v(A-B)\}=\v(A),$$ and hence obtain $\v(A)= \v(B)$.

	\noindent	(2) From $A\sim B$, we have $\v(A)<\v(A-B)$. Hence, $\v(AC)<\v(AC-BC)$, which completes the proof.
	
	\noindent	(3) 	 Reflexivity: Trivial.
	
	\noindent Symmetry: For any $A, B$ in $\calK$ with $A\sim B$, by (1) we have $\v(A)=\v(B)$ and hence $B\sim A$.
	
	\noindent Transitivity:  For any $A, B, C$ in $\calK$ with $A\sim B\sim C$, we have $\v(A-B)>\v(A)$ and $\v(B-C)>\v(B)$. Combined with (1) that $\v(A)=\v(B)$, this implies that
	$$\v(A-C)\geq \min\{\v(A-B),\v(B-C)\}>\v(A),$$ and hence $A\sim C$. \qedhere
\end{proof}
Recall that $\mu=1-\lambda$. 
\begin{lemma}\label{mu}
	\noindent
	\begin{enumerate}
		\item We have $\lambda\sim 1$.
		\item 	
		Given any integer $r\geq0$, we have 
		$$1-\lambda^r\sim \frac{r}{p^{\vp(r)}} \mu^{p^{\vp(r)}}.$$
	\end{enumerate}
\end{lemma}
\begin{proof}
	(1) Note that from our assumption $0<|1-\lambda|<1$, we have 
	$\v(1-\lambda)>0=\v(1)$, and hence $\lambda\sim 1$.
	
	(2) Let $\ell:=\vp(r)$. We have
	\[1-\lambda^{r}=1-(1-(1-\lambda^{p^\ell}))^{\frac{r}{p^\ell}}=-\sum_{i=1}^{\frac{r}{p^\ell}}\binom{\frac{r}{p^\ell}}{i}_p(-\mu^{p^\ell})^i\sim \frac{r}{p^\ell}\mu^{p^\ell}.\qedhere\]
\end{proof}

\begin{lemma}\label{new lemma}
	Assume $T>1.$  Given any $k\geq 1$ such that $p^{k}\leq T$ and $p^k$-divisible $r\geq0$, we have 	\begin{equation*}
		\phi_k(r, r+p^k)\sim \frac{2a_1a_{p-1}^{\frac{p^k-1}{p-1}}}{1-\lambda^{r+p^k}}\prod_{i=0}^{k-1}\frac{1}{\mu^{p^i}}.
	\end{equation*}	
\end{lemma}
\begin{proof}

	With the assumptions in this lemma, we first prove by induction that for every $p^k$-divisible $r$ and $0\leq \ell\leq k$ we have
	\begin{equation}\label{12d}
		\phi_\ell(r, r+p^\ell)\sim\frac{a_1}{\lambda(1-\lambda^{r+1})}\prod_{i=0}^{\ell-1}\frac{\lambda^{r+1}a_{p-1}^{p^i}}{\lambda(1-\lambda^{r+p^{i+1}})}\binom{r+p^{i}+1}{\ualpha(r+p^{i},r+p^{i+1})}_p.
	\end{equation}
	For $\ell=0$, from our assumption $p^k|r$ and $k\geq 1$ we have $$\phi_0(r, r+1)=\frac{a_1}{\lambda(1-\lambda^{r+1})}\binom{1+r}{r,1,0}_p=\frac{a_1}{\lambda(1-\lambda^{r+1})},$$
	and hence prove \eqref{12d}. 
	
	Now suppose that \eqref{12d} holds for an arbitrarily given $0\leq \ell\leq k-1$. Consider \begin{equation}
		\phi_{\ell+1}(r,r+p^{\ell+1})=\sum_{\ubeta\in S_1(r/p^\ell,r/p^\ell+p)}\prod_{j=0}^{m(\ubeta)}\phi_\ell(\beta_jp^\ell,\beta_{j+1}p^\ell).
	\end{equation}
	For each $\ubeta\in S_1(r/p^\ell,r/p^\ell+p)$, there exists $0\leq j'\leq m(\ubeta)$ such that $p-1\nmid \beta_{j'+1}-\beta_{j'}$. 
	By Lemma~\ref{key lemma}(1), 
	\begin{equation}\label{18e}
		\v(\phi_\ell(\beta_{j'}p^\ell, \beta_{j'+1}p^\ell))\geq M_0(\beta_{j'+1}-\beta_{j'})p^\ell+T-p^{\ell+\vp(\beta_{j'+1})},
	\end{equation}
	with equality if and only if $\begin{cases}
		\beta_{j'+1}-\beta_{j'}=1 \textrm{~and~} p\nmid \beta_{j'}+1& \textrm{for~} \ell=0,\\
		\beta_{j'+1}-\beta_{j'}=1& \textrm{for~} \ell\geq 1.
	\end{cases}$
	
	For every $j\in \{0,\dots,m(\ubeta)\}\backslash \{j'\}$, by Lemma~\ref{key lemma}(3) on $\ell\leq k-1$ we have 
	\begin{equation}\label{18d}
		\v(\phi_\ell(\beta_{j}p^\ell, \beta_{j+1}p^\ell))\geq M_0(\beta_{j+1}-\beta_{j})p^\ell+p^\ell-p^{\ell+\vp(\beta_{j+1})},
	\end{equation}
	with equality if and only if  $\beta_{j+1}-\beta_j=p-1$ and 
	$
	\begin{cases}
		p\nmid \beta_{j}+1 &\textrm{for~}\ell=0, \\
		p\nmid \beta_{j}  &\textrm{for~} \ell\geq 1.
	\end{cases}
	$
	
	In particular, 
	\begin{equation}\label{18c}
		\v\left(\phi_\ell(r+p^\ell, r+p^{\ell+1})\right)=M_0(p-1)p^\ell+p^\ell-p^{\vp(r+p^{\ell+1})}.
	\end{equation}
	
	From \eqref{18e} and \eqref{18d}, we have
	$$\v\left(\prod_{j=0}^{m(\ubeta)}\phi_\ell(\beta_jp^\ell,\beta_{j+1}p^\ell)\right)\geq M_0p^{\ell+1}+T-p^{\vp(r+p^{\ell+1})},$$
	with equality if and only if $\ubeta=(r/p^\ell,r/p^\ell+1,r/p^\ell+p)$; and hence
	\begin{equation}\label{14b}
		\phi_{\ell+1}(r,r+p^{\ell+1})\sim\phi_\ell(r,r+p^\ell)\phi_\ell(r+p^\ell, r+p^{\ell+1}).
	\end{equation}
	

	Combining  Lemma~\ref{key lemma}(2) with $T\geq p^\ell$, we have
	\begin{multline*}
		\v\left(\phi_\ell(r+p^\ell, r+p^{\ell+1})-\frac{\lambda^{r+1}a_{p-1}^{p^\ell}}{\lambda(1-\lambda^{r+p^{\ell+1}})}\binom{r+p^{\ell}+1}{\ualpha(r+p^{\ell},r+p^{\ell+1})}_p\right)\\
		>
		M_0(p-1)p^\ell+T-p^{\vp(r+p^{\ell+1})}
		\geq M_0(p-1)p^\ell+p^\ell-p^{\vp(r+p^{\ell+1})}.
	\end{multline*}
	Together with \eqref{18c}, this gives us
	$$\phi_\ell(r+p^\ell, r+p^{\ell+1})\sim \frac{\lambda^{r+1}a_{p-1}^{p^\ell}}{\lambda(1-\lambda^{r+p^{\ell+1}})}\binom{r+p^{\ell}+1}{\ualpha(r+p^{\ell},r+p^{\ell+1})}_p.$$
	Combined with Lemma~\ref{tri}(2) and \eqref{14b}, this similarity relation implies
	\begin{equation}\label{14f}
		\phi_{\ell+1}(r,r+p^{\ell+1})\sim \phi_\ell(r,r+p^\ell)\frac{\lambda^{r+1}a_{p-1}^{p^\ell}}{\lambda(1-\lambda^{r+p^{\ell+1}})}\binom{r+p^{\ell}+1}{\ualpha(r+p^{\ell},r+p^{\ell+1})}_p.
	\end{equation}

	Finally, by our induction hypothesis for $\ell$
	and Lemma~\ref{tri}(2) again, we prove \eqref{12d} for $\ell+1$, and complete the induction.
	
	Now we prove this lemma. Based on our assumption $p^k|r$, for $0\leq \ell\leq k-1$ by Lemma~\ref{mu}(2), we have 
	$$1-\lambda^{r+p^\ell}\sim (r/p^\ell+1)\mu^{p^\ell}= \mu^{p^\ell};$$
	and by Lemma~\ref{alpha},
	$$\binom{r+p^{\ell}+1}{\ualpha(r+p^{\ell},r+p^{\ell+1})}_p=\begin{cases}
		r+1+1=2 &\textrm{for~} \ell=0,\\
		\frac{r+p^{\ell}}{p^\ell}=1&\textrm{for~} 1\leq\ell\leq k-1.
	\end{cases}$$
	Plugging them back in \eqref{12d} for $k$ and using  Lemma~\ref{tri}(1), we obtain
	\begin{equation*}
		\phi_k(r, r+p^k)\sim \frac{2a_1a_{p-1}^{\frac{p^k-1}{p-1}}}{1-\lambda^{r+p^k}}\prod_{i=0}^{k-1}\frac{1}{\mu^{p^i}},
	\end{equation*}	
	and hence complete the proof.
\end{proof}

\begin{lemma}\label{Lemma10a}
	Given any integers $0\leq r<s$ and any monomial $a_1^{\gamma_1}a_{p-1}^{\gamma_{p-1}}$ (as a polynomial of $a_1$ and $a_{p-1}$) in $\phi_1(r,s)$, we have $$\gamma_1+(p-1)\gamma_{p-1}=s-r.$$ 
\end{lemma}
\begin{proof}
	Consider $\phi_1(r,s)=\sum_{\ubeta\in S_1(r,s)}\Phi(\ubeta)$. 	Note that for every $\ubeta\in S_1(r,s)$, $\Phi(\ubeta)$ is the sum of terms of the form $*\prod_{j=0}^{m(\ubeta)}a_1^{\alpha^{(j)}_1}a_{p-1}^{\alpha^{(j)}_{p-1}}$, where $*\in \calK$ and $\ualpha^{(j)}\in I'(\beta_j,\beta_{j+1})$ for $0\leq j\leq m(\ubeta)$.
	 From  $\alpha^{(j)}_1+(p-1)\alpha^{(j)}_{p-1}=\beta_{j+1}-\beta_j$ for every $0\leq j\leq m(\ubeta)$, we have $$\sum_{j=0}^{m(\ubeta)}\alpha^{(j)}_1+(p-1)\sum_{j=0}^{m(\ubeta)}\alpha^{(j)}_{p-1}=\sum_{j=0}^{m(\ubeta)}(\beta_{j+1}-\beta_j)=s-r.$$
	Clearly, this equality is independent to the choices of $\ubeta$ and $\{\ualpha^{(j)}\}_{j=0}^{m(\ubeta)}$. Hence, we complete the proof.
\end{proof}

\begin{lemma}\label{Lemma10b}
	For any $0\leq r<s$ and $\ubeta\in S_1(r,s)$, let $A\in \calK$ be the coefficient of  $a_1^{\gamma_1}a_{p-1}^{\gamma_{p-1}}$ in $\Phi(\ubeta)$, then 
	$$\v(Aa_1^{\gamma_1}a_{p-1}^{\gamma_{p-1}})\geq \gamma_1\v(a_1)+\gamma_{p-1}\v(a_{p-1})-m(\ubeta)-p^{\vp(s)}.$$
\end{lemma}
\begin{proof}	
	Note that $Aa_1^{\gamma_1}a_{p-1}^{\gamma_{p-1}}$  is a sum of terms of the form 
	\begin{equation*}\label{10a'}
		\prod_{j=0}^{m(\ubeta)} \frac{1}{\lambda(1-\lambda^{\beta_{j+1}})}\binom{\beta_{j}+1}{\ualpha^{(j)}}_p \lambda^{\alpha^{(j)}_0}a_1^{\alpha^{(j)}_1}a_{p-1}^{\alpha^{(j)}_{p-1}},	
	\end{equation*}
	where $\{\ualpha^{(j)}\}_{j=0}^{m(\ubeta)}$ satisfies $$\ualpha^{(j)}\in I'(\beta_j,\beta_{j+1}), \ \sum_{j=0}^{m(\ubeta)} \alpha^{(j)}_1=\gamma_1\textrm{~and~}\sum_{j=0}^{m(\ubeta)} \alpha^{(j)}_{p-1}=\gamma_{p-1}.$$
	
	Note that $\v(\lambda)=0$ and $\v(1-\lambda^{\beta_{j+1}})=\begin{cases}
		p^{\vp(s)}&\textrm{for~} \beta_{j+1}= s\\
		1&\textrm{otherwise}
	\end{cases}.$
	We have 
	\begin{multline*}
		\v\left( \prod_{j=0}^{m(\ubeta)} \frac{1}{\lambda(1-\lambda^{\beta_{j+1}})}\binom{\beta_{j}+1}{\ualpha^{(j)}}_p \lambda^{\alpha^{(j)}_0}a_1^{\alpha^{(j)}_1}a_{p-1}^{\alpha^{(j)}_{p-1}}\right) \\
		\geq \v(a_1)\sum_{j=0}^{m(\ubeta)} \alpha^{(j)}_1+\v(a_{p-1})\sum_{j=0}^{m(\ubeta)} \alpha^{(j)}_{p-1}-m(\ubeta)-p^{\vp(s)}\\
		=\gamma_1\v(a_1)+\gamma_{p-1}\v(a_{p-1})-m(\ubeta)-p^{\vp(s)}.
	\end{multline*}
	This completes the proof. 
\end{proof}

\begin{proof}[Proof of Theorem~\ref{Thm:p-1}]
	
	As noted in the beginning of the section, we only need to prove that $f(z)=\lambda z+a_1 z^{2}+a_{p-1}z^{p}\in\mathcal{K}[z]$ is non-linearizable if $a_1\neq 0$.
	By Criterion~\ref{c:key criterion}, it is enough to show that it is $k$-dominant for some $k\geq 1$.
	
	Note that by Lemma~\ref{l:0-th}, \[M_0=\begin{cases}
		\v(a_1)-1&\textrm{if~}T<1,\\
		\frac{\v(a_{p-1})-1}{p-1}& \textrm{if~} T\geq 1.
	\end{cases}\]		
	
	According to the ranges of $T$, we split our discussion into three cases.
	
	\noindent\textbf{Case I:} $T\leq 1$ and $T\neq \frac{p-2}{p-1}$. 
	
	By Lemma~\ref{Lemma10a},  \begin{equation}\label{15c}
		\phi_1(0,p)=A_1a_1a_{p-1}+B_1a_1^p \textrm{~for some~} A_1\in \calK \textrm{~and~}  B_1\in \calK.
	\end{equation}
	
	We first study $A_1a_1a_{p-1}$. Note that 
	\begin{equation}\label{11e}
		\phi_1(0,p)=\sum_{\ubeta\in S_1(0, p)} \Phi(\ubeta);
	\end{equation}
	and that $A_1a_1a_{p-1}$ can only be obtained from $\ubeta=(0,p), (0,p-1,p)$ and $(0,1,p)$. From $I'(0,p)=\emptyset$ and Lemma~\ref{5.3}, the contributions from $(0,p)$ and $(0,p-1,p)$ are both $0$. Therefore, we obtain
	\begin{equation}\label{16a}
		A_1a_1a_{p-1}=\Phi((0,1,p))=\Phi(0,1)\Phi(1,p)
	=\frac{1}{\lambda(1-\lambda)}a_1\times \frac{2}{\lambda(1-\lambda^p)}\lambda a_{p-1},
	\end{equation}
and
	\begin{equation}\label{10h}
		\v(A_1a_1a_{p-1})=\v(a_1)+\v(a_{p-1})-1-p. 
	\end{equation}
	We next study $B_1a_1^p$. 	By Lemma~\ref{5.3}, for $\ubeta\in S_1(0,p)$ if $p-1\in \ubeta$ then $\Phi(\ubeta)=0$. Therefore, for $\ubeta\in S_1(0,p)$
	with $\Phi(\ubeta)\neq 0$, we have $m(\ubeta)\leq p-2$ with equality if and only if $\ubeta$ is equal to $\ubeta':=(0,1,\dots,p-2,p)$.
	Combining it with Lemma~\ref{Lemma10b}, we conclude that for every  $\ubeta\in S_1(0,p)\backslash \{\ubeta'\}$ the monomial $*a_1^p$ in $\Phi(\ubeta)$ is either $0$ or satisfies \begin{equation}\label{23r}
		\v(*a_1^p)\geq p\v(a_1)-m(\ubeta)-p\geq p\v(a_1)-2p+3. 
	\end{equation}

	For $\ubeta'$, we have
	\begin{multline*}
		\Phi(\ubeta')=
		\frac{1}{\lambda(1-\lambda^{p})}\binom{p-1}{p-3, 2, 0}_p \lambda^{p-3}a_1^2\times  	\prod_{j=0}^{p-3} \frac{1}{\lambda(1-\lambda^{j+1})}\binom{j+1}{j, 1, 0}_p \lambda^{j}a_1
		\\=a_1^p\frac{\lambda^{p-4}}{1-\lambda^{p}} \times  	\prod_{j=0}^{p-3} \frac{\lambda^{j-1}}{1-\lambda^{j+1}}(j+1),
	\end{multline*}
	and hence \begin{equation*}
	\v(	\Phi(\ubeta'))=	\v\left(	a_1^p\frac{\lambda^{p-4}}{1-\lambda^{p}} \times  	\prod_{j=0}^{p-3} \frac{\lambda^{j-1}}{1-\lambda^{j+1}}(j+1) \right)=p\v(a_1)-2p+2.
	\end{equation*}
Combining this equality with \eqref{23r}, we obtain
	\begin{equation*}
		\v(B_1a_1^p-	\Phi(\ubeta'))\geq p\v(a_1)-2p+3>\v(\Phi(\ubeta')),
	\end{equation*}
and hence
	$$\v(B_1a_1^p)=\v(\Phi(\ubeta'))=p\v(a_1)-2p+2.$$
	Combined with \eqref{10h} and our assumption $T\neq \frac{p-2}{p-1}$ in this case, this equality implies that $$\v(B_1a_1^p)\neq 	\v(A_1a_1a_{p-1}).$$
	Together with \eqref{15c} and $M_0=\v(a_1)-1$, this above inequality implies that 
	 $$M_1(0,p)\leq \frac{\v(B_1a_1^p)}{p}=\frac{p\v(a_1)-2p+2}{p}=M_0-\frac{p-2}{p}\leq M_0-\frac{1}{p},$$
	where the last inequality is from our assumption $p\geq 5$. This shows that $f$ is $1$-dominant.
	
%
%
%
	
	\noindent\textbf{Case II:} $T=\frac{p-2}{p-1}$.
	
	Note that in this case, we have	
	\begin{equation}\label{17j}
		\v(a_{p-1})=(p-1)\v(a_1)-p+3
	\end{equation}
	and 
	$$\v(A_1a_1a_{p-1})=	\v(B_1a_1^p)=	p\v(a_1)-2p+2.$$
	
	If $A_1a_1a_{p-1}+B_1a_1^p=p\v(a_1)-2p+2,$
	by \eqref{15c} we have 
	$$\phi_1(0,p)=p\v(a_1)-2p+2,$$
	and hence 
	$$M_1(0,p)=\v(a_1)-2+\frac{2}{p}=M_0-1+\frac{2}{p}
	\leq  M_0-\frac{1}{p}.$$
	
	Now we assume $A_1a_1a_{p-1}+B_1a_1^p>p\v(a_1)-2p+2,$
	i.e.
	$A_1a_1a_{p-1}\sim -B_1a_1^p.$
	Note that in this case, $\v(\phi_1(0,p))$ can be arbitrarily large, so it is not feasible to prove by $M_1(0,p)\leq M_0-\frac{1}{p}$ that $f$ is $1$-dominant. Instead, we use $\phi_1(0,2p)$ for this case. 
	
	By Lemma~\ref{Lemma10a},  $$\phi_1(0,2p)=A_2a_1^2a_{p-1}^2+B_2a_1^{2p}+C_2a_1^{p+1}a_{p-1}$$
	for some $A_2, B_2, C_2$ in $\calK.$

	We first study $A_2a_1^2a_{p-1}^2$ with the consideration of 
	\begin{equation}\label{11e'}
		\phi_1(0,2p)=\sum_{\ubeta\in S_1(0,2p)} \Phi(\ubeta).
	\end{equation}

	Note that every $\ubeta\in S_1(0,2p)$ such that $\Phi(\ubeta)$ has nonzero $a_1^2a_{p-1}^2$ term satisfies \begin{equation}\label{23t}
		m(\ubeta)\leq 3;
	\end{equation} and that by Lemma~\ref{5.3}, every $\ubeta\in S_1(0,2p)$ with $\Phi(\ubeta)\neq 0$ satisfies $\beta_j\neq p-1$, $p$ or $2p-1.$   Hence, the equality in \eqref{23t} is achieved uniquely by
	$\ubeta':=(0,1,2,p+1,2p)$.
	
	By Lemma~\ref{Lemma10b}, for every  $\ubeta\in S_1(0,2p)\backslash \{\ubeta'\}$ the monomial $*a_1^2a_{p-1}^2$ in $\Phi(\ubeta)$ is either $0$ or satisfies \begin{equation}\label{19a}
		\v(*a_1^2a_{p-1}^2)\geq 2\v(a_1)+2\v(a_{p-1})-m(\ubeta)-p\geq 2\v(a_1)+2\v(a_{p-1})-2-p. 
	\end{equation}
	
	On the other hand, 
	\begin{align*}
		\Phi(\ubeta')=&\frac{1}{\lambda(1-\lambda)}\binom{1}{0,1,0}_pa_1\times \frac{1}{\lambda(1-\lambda^2)}\binom{2}{1,1,0}_p\lambda a_1
		\\
		&	\quad\quad\times \frac{1}{\lambda(1-\lambda^{p+1})}\binom{3}{2,0,1}_p\lambda^2 a_{p-1}\times \frac{1}{\lambda(1-\lambda^{2p})}\binom{p+2}{p+1,0,1}_p\lambda^{p+1}a_{p-1}+*'a_1^{2p}
		\\=&\frac{12\lambda^{p}}{(1-\lambda)(1-\lambda^2)(1-\lambda^{p+1})(1-\lambda^{2p})}a_1^2a_{p-1}^2+*'a_1^{2p},
	\end{align*}
	where $*'\in \calK$. 
	From our assumption $p\geq 5$, we have $$\frac{12\lambda^{p}}{(1-\lambda)(1-\lambda^2)(1-\lambda^{p+1})(1-\lambda^{2p})}\neq 0,$$
	and hence $$\v\left(\frac{12\lambda^{p}}{(1-\lambda)(1-\lambda^2)(1-\lambda^{p+1})(1-\lambda^{2p})}a_1^2a_{p-1}^2\right)=2\v(a_1)+2\v(a_{p-1})-3-p.$$
	
	Combined with \eqref{19a}, this implies
	\begin{equation*}
		\v(A_2a_1^2a_{p-1}^2)=2\v(a_1)+2\v(a_{p-1})-3-p;
	\end{equation*}
	further with \eqref{17j},  
	\begin{equation}\label{11f}
		\v(A_2a_1^2a_{p-1}^2)=2p\v(a_1)-3p+3.
	\end{equation}
	
We next estimate $B_2a_1^{2p}$.  	By Lemma~\ref{5.3}, every $\ubeta\in S_1(0,2p)$ with $\Phi(\ubeta)\neq 0$ satisfies $\beta_j\neq p-1$, $p$ or $2p-1$, and consequently $m(\ubeta)\leq 2(p-2)$.
%
%
	By Lemma~\ref{Lemma10b}, for every  $\ubeta\in S_1(0,2p)$ the monomial $*a_1^{2p}$ in $\Phi(\ubeta)$ is either $0$ or satisfies \begin{equation}\label{19b}
		\v(*a_1^{2p})\geq 2p\v(a_1)-m(\ubeta)-p\geq 2p\v(a_1)-3p+4. 
	\end{equation}

	This implies
	\begin{equation}\label{17l}
		\v(B_2a_1^{2p})\geq 2p\v(a_1)-3p+4.
	\end{equation}
	
	Now we study $C_2a_1^{p+1}a_{p-1}$. 
		Note that for every $\ubeta\in S_1(0,2p)$ such that $\Phi(\ubeta)$ has nonzero $a_1^{p+1}a_{p-1}$, there exists $0\leq j'\leq m(\ubeta)$ such that $\beta_{j'+1}-\beta_{j'}\geq p-1$.
	Assuming in addition that $\underline{\beta}$ does not contain $p-1$ or $2 p-1$, so that by Lemma~\ref{5.3} the number $\Phi(\beta)$ is nonzero.
		If  $\beta_{j'+1}-\beta_{j'}\geq p+1$, then clearly $m(\ubeta)\leq p-1$.
		If $p-1\leq \beta_{j'+1}-\beta_{j'}\leq p$, the interval $(\beta_{j'},\beta_{j'+1})$ contains at most one of $p-1$ and $2p-1$. Hence, we have 
		$$m(\beta)\leq 2p-(\beta_{j'+1}-\beta_{j'})-1\leq p$$
		with equality if and only if $\ubeta$ is equal to one of the following sequences \begin{equation}\label{17g}
			\ubeta^{(j)}:=(0, 1, \dots, j , j+p-1, \dots, 2p-2, 2p)\in S_1(0,2p) \textrm{~for every~} 2\leq j\leq p-2.
		\end{equation}

%
%
%
%

%
%
%
%
%
%
%
%
%
%
%
	We first determine the $a_1^{p+1}a_{p-1}$  term in $\Phi(\ubeta^{(j)})$
	for each $2\leq j\leq p-2$ by 
	\begin{align*}
		\Phi(\ubeta^{(j)})=&\prod_{i=0}^{j-1}\Phi(i,i+1)\times \Phi(j,j+p-1)\times \prod_{\ell=j+p-1}^{2p-3}\Phi(\ell,\ell+1)\times \Phi(2p-2, 2p)\\
		=& \prod_{i=0}^{j-1}\frac{(i+1)\lambda^i a_1}{\lambda(1-\lambda^{i+1})}\times \frac{(j+1)\lambda^ja_{p-1}}{\lambda(1-\lambda^{j+p-1})}\times \prod_{\ell=j+p-1}^{2p-3}\frac{(\ell+1)\lambda^\ell a_1}{\lambda(1-\lambda^{\ell+1})}\times \frac{\binom{2p-1}{2p-3,2,0}_p\lambda^{2p-3}a_1^2}{\lambda(1-\lambda^{2p})}.
	\end{align*}

	By Lemmas~\ref{tri}(2) and~\ref{mu}, we have
	\begin{equation*}\label{17h}
		\Phi(\ubeta^{(j)})\sim \prod_{i=0}^{j-1}\frac{a_1}{\mu}\times \frac{(j+1)a_{p-1}}{(j-1)\mu}\times \prod_{\ell=j+p-1}^{2p-3}\frac{a_1}{\mu}\times \frac{a_1^2}{2\mu^p}=
		\frac{a_1^{p+1}a_{p-1}}{2\mu^{2p}}\times \frac{j+1}{j-1},
	\end{equation*}
	and hence for every $2\leq j\leq p-2$, 
	\begin{multline*}
		\v\left(\Phi(\ubeta^{(j)})-\frac{a_1^{p+1}a_{p-1}}{2\mu^{2p}}\times \frac{j+1}{j-1}\right)>\v\left(	\frac{a_1^{p+1}a_{p-1}}{2\mu^{2p}}\times \frac{j+1}{j-1}\right)
		\\=  (p+1)\v(a_1)+\v(a_{p-1})-2p. 
	\end{multline*}
	
	Hence, we have
	\begin{equation*}\label{17i}
		\v\left(	\sum_{j=2}^{p-2}\Phi(\ubeta^{(j)})-	\frac{a_1^{p+1}a_{p-1}}{2\mu^{2p}}\times \sum_{j=2}^{p-2}\frac{j+1}{j-1}\right)> (p+1)\v(a_1)+\v(a_{p-1})-2p. 
	\end{equation*}
	
	By calculation, \begin{multline*}
		\sum_{j=2}^{p-2}\frac{j+1}{j-1}=\sum_{j=1}^{p-1}\frac{j+2}{j}-\frac{p}{p-2}-\frac{p+1}{p-1}
		=\frac{1}{2}\left(\sum_{j=1}^{p-1} \frac{p-j+2}{p-j}+\sum_{j=1}^{p-1} \frac{j+2}{j}\right)+1\\
		=\frac{1}{2}\left(\sum_{j=1}^{p-1} \frac{j-2}{j}+\sum_{j=1}^{p-1} \frac{j+2}{j}\right)+1=0.
	\end{multline*}

	Hence, we have
	\begin{equation}\label{17k}
		\v\left(\sum_{j=2}^{p-2}\Phi(\ubeta^{(j)})\right)>(p+1)\v(a_1)+\v(a_{p-1})-2p.
	\end{equation}
	%
	%
	%
	%
	%
	%

	By Lemma~\ref{Lemma10b}, for $\ubeta\in S_1(0,2p)$ not in the list \eqref{17g}, we have either $\Phi(\ubeta)=0$ or the $\mu$-adic valuation of the $a_1^{p+1}a_{p-1}$ term in $\Phi(\ubeta)$ is greater or equal to
	$$ (p+1)\v(a_1)+\v(a_{p-1})-m(\ubeta)-p\geq (p+1)\v(a_1)+\v(a_{p-1})-2p+1. $$
	
	Combined with \eqref{17k}, this implies 
	$$\v(C_2a_1^2a_{p-1}^2)>(p+1)\v(a_1)+\v(a_{p-1})-2p;$$
	and further with \eqref{17j}, 
	$$\v(C_2a_1^2a_{p-1}^2)>2p\v(a_1)-3p+3.$$
	
	From \eqref{11f}, \eqref{17l} and the strict inequality above, we have 
	$$\v(\phi_1(0,2p))=\v(A_2a_1^2a_{p-1}^2)=2p\v(a_1)-3p+3.$$
	
	Note that $M_0=\v(a_1)-1$. Thus, 
	$$M_1(0,2p)=\frac{2p\v(a_1)-3p+3}{2p}=M_0-\frac{p-3}{2p}\leq M_0-\frac{1}{p},$$
	where the last inequality is from our assumption $p\geq 5$. This shows that $f$ is $1$-dominant.

	\noindent\textbf{Case III:} $T>1$.
	
	Let $k\geq 1$ be the integer such that $p^{k-1}<T\leq p^{k}$. We are going to show that $f$ is $(k+1)$-dominant. To this end, we will use several times that $M_{0}=\frac{\mathrm{val}_{\mu}\left(a_{p-1}\right)-1}{p-1}$, which follows from the definition of $T$ and from Lemma~\ref{l:0-th}.
	We also note that for any given $1\leq \ell\leq k-1$ and $p^\ell$-divisible integers $0\leq r<s$ we have $$M_\ell(r, s)=\frac{\v(\psi_\ell(r,s))}{s-r}=\frac{\v(\phi_\ell(r,s))+p^{\vp(s)}-p^\ell}{s-r};$$
	and that by Lemma~\ref{key lemma}(3), 
	$M_\ell(r, s)\geq M_0$
	with equality if and only if  $s-r=(p-1)p^\ell$ and 	$p^{\ell+1}\nmid r$. 
	Hence, we have \begin{equation}\label{12b}
		M_\ell=M_0 \textrm{~for every~} 1\leq \ell\leq k-1.
	\end{equation}

	Recall that for any integers $0\leq r<s$ with $p-1|s-r$ we put  $\ualpha(r,s):=(r+1-\frac{s-r}{p-1},0,\frac{s-r}{p-1})\in I'(r,s)$.
	Now for any $p^k$-divisible integers $0\leq r<s$ with $p-1|s-r$, we have 
	\begin{multline*}
		\v\left(\frac{1}{\lambda(1-\lambda^s)}\binom{r+1}{\ualpha(r,s)}_p\ua^{\ualpha(r,s)}\right)
		\geq \frac{s-r}{p-1}\v(a_{p-1})-p^{\vp(s)}\\
		=(s-r)M_0+\frac{s-r}{p-1}-p^{\vp(s)}\geq (s-r)M_0+p^k-p^{\vp(s)}\\
		\geq (s-r)M_0+T-p^{\vp(s)},
	\end{multline*}
	where the equality conditions for the three inequalities are $\binom{r+1}{\ualpha(r,s)}_p\neq0$, $s-r=(p-1)p^k$ and $T=p^k$, respectively. 
	
	Note that if $s=r+(p-1)p^k$, by Lemma~\ref{alpha} we have $\binom{r+1}{\ualpha(r,s)}_p=\frac{r}{p^k}$, and hence the equality conditions can be further simplified to $s-r=(p-1)p^k$, $p^{k+1}\nmid r$ and $T=p^k$. 
	
	Therefore, by Lemma~\ref{key lemma}(2), for any $p^k$-divisible $0\leq r<s$ with $p-1|s-r$ we have
	\begin{equation*}
		\v(\phi_k(r,s))\geq (s-r)M_0+T-p^{\vp(s)},
	\end{equation*} 
	with equality if and only if $s-r=(p-1)p^k$, $p^{k+1}\nmid r$ and $T=p^k$.

	 Combined with Lemma~\ref{key lemma}(1), this implies that for any $p^k$-divisible $0\leq r<s$, 
	\begin{equation}\label{k}
		\v(\phi_k(r,s))\geq (s-r)M_0+T-p^{\vp(s)},
	\end{equation} 
	with equality if and only if $$\begin{cases}
		s-r=p^k& \textrm{if~} T<p^k,\\
		s-r=p^k,\ \textrm{or~} s-r=(p-1)p^k \textrm{~and~} p^{k+1}\nmid r& \textrm{if~} T=p^k;\\
	\end{cases}$$
	and
$$M_k(r, s)=\frac{\v(\psi_k(r,s))}{s-r}=\frac{\v(\phi_k(r,s))+p^{\vp(s)}-p^k}{s-r}\geq M_0+\frac{T-p^k}{s-r}\geq M_0+\frac{T-p^k}{p^k},$$
	with equality if $s-r=p^k$.
As a consequence, we have \begin{equation}\label{19c}
		M_k=M_0+\frac{T-p^k}{p^k}.
	\end{equation}

	Now we are going show that either
	$$
	M_{k+1}\left(0, p^{k}\right) \text { or } M_{k+1}\left(0,2 p^{k}\right)
	$$
	is less than or equal to $M_{k}-\frac{1}{p}$. Together with \eqref{12b} and \eqref{19c}, this implies that $f$ is $(k+1)$-dominant. To achieve it, we further split our discussion into two subcases. 

	\noindent\textbf{(a)} $T<p^k$. 
	Consider	\begin{equation}\label{pk+1}
		\phi_{k+1}(0,p^{k+1})=\sum_{\ubeta\in S_1(0, p)}\prod_{j=0}^{m(\ubeta)}\phi_k(\beta_jp^k,\beta_{j+1}p^k),
	\end{equation}

	By \eqref{k}, for every $\ubeta\in S_1(0,p)$ we have
	\begin{align*}
		\v\left(  \prod_{j=0}^{m(\ubeta)}\phi_k(\beta_jp^k,\beta_{j+1}p^k) \right)
		\geq &\sum_{j=0}^{m(\ubeta)}\left((\beta_{j+1}-\beta_j)p^kM_0+T-p^{\vp(\beta_{j+1})+k}\right)\\
		=&p^{k+1}M_0+(m(\ubeta)+1)T-m(\ubeta)p^k-p^{k+1}\\
		=&p^{k+1}M_0+(m(\ubeta)+1)(T-p^k)+p^k-p^{k+1}\\	
		\geq &p^{k+1}M_0+p(T-p^k)+p^k-p^{k+1},
	\end{align*}
	with equality if and only if 
	$\ubeta=(0,1,\dots,p)$.
	
	By \eqref{pk+1}, this implies
	$$\v(	\phi_{k+1}(0,p^{k+1}))=p^{k+1}M_0+p(T-p^k)+p^k-p^{k+1}.$$
	
	Hence, by \eqref{19c} we have  $$M_{k+1}(0,p^{k+1})=M_0+\frac{p(T-p^k)+p^k-p^{k+1}}{p^{k+1}}\leq M_k-\frac{1}{p}.$$
As noted above, this proves that $f$ is $(k+1)$-dominant.
	
	\noindent\textbf{(b)} $T=p^k$. 
	
	Note that in this case, we have $M_k=M_0$.
	Similar to the reason mentioned in Case II, we need to  study both $\phi_{k+1}(0,p^{k+1})$ and $\phi_{k+1}(0,2p^{k+1})$ in this case. 
	%
	%
	Consider \eqref{pk+1} again.
	From \eqref{k} and our assumption $T=p^k$, for every $\ubeta\in S_1(0,p)$ we have
	\begin{multline}\label{15b}
		\v\left(  \prod_{j=0}^{m(\ubeta)}\phi_k\left(\beta_jp^k,\beta_{j+1}p^k\right) \right)
		\geq \sum_{j=0}^{m(\ubeta)}\left((\beta_{j+1}-\beta_j)p^kM_0+p^k-p^{\vp(\beta_{j+1})+k}\right)\\
		=p^{k+1}M_0+p^k-p^{k+1},
	\end{multline}
	with equality if and only if for every $0\leq j\leq m(\ubeta)$ we have
	$\beta_{j+1}-\beta_j=1$; or $\beta_{j+1}-\beta_j=p-1$ and $p\nmid \beta_{j}$.
	Note that this equality can only be  achieved by $\phi_k(0,p^k)\phi_k(p^k, p^{k+1})$ and $\prod_{j=0}^{p-1} \phi_k\left(jp^k,(j+1)p^k\right)$ for $(0,1,p)$ and $(0,1,\dots,p)$, i.e.
	$$\v\left(\phi_k(0,p^k)\phi_k(p^k, p^{k+1})\right)=\v\left(\prod_{j=0}^{p-1} \phi_k\left(jp^k,(j+1)p^k\right)\right)=p^{k+1}M_0+p^k-p^{k+1}.$$
	
	Therefore, we have
	$$\v\left( \phi_{k+1}(0,p^{k+1})-\phi_k(0,p^k)\phi_k(p^k, p^{k+1})-\prod_{j=0}^{p-1} \phi_k\left(jp^k,(j+1)p^k\right)\right)>p^{k+1}M_0+p^k-p^{k+1}.$$
	
	If $$\v\left(\phi_k(0,p^k)\phi_k(p^k, p^{k+1})+\prod_{j=0}^{p-1} \phi_k\left(jp^k,(j+1)p^k\right)\right)=p^{k+1}M_0+p^k-p^{k+1},$$
	then $$\v\left( \phi_{k+1}(0,p^{k+1})\right)=p^{k+1}M_0+p^k-p^{k+1},$$
	and hence 
	\begin{equation}\label{19d}
		M_{k+1}(0,p^{k+1})=M_0+\frac{1-p}{p}\leq M_0-\frac{1}{p}.
	\end{equation}
	
	From \eqref{12b} and \eqref{19c}, we have $M_\ell=M_0$ for every $0\leq \ell\leq k$. Therefore, \eqref{19d} implies that $f$ is $(k+1)$-dominant.
	
	Now we assume \begin{equation}\label{15a}
		\v\left(\phi_k(0,p^k)\phi_k(p^k, p^{k+1})+\prod_{j=0}^{p-1} \phi_k\left(jp^k,(j+1)p^k\right)\right)>p^{k+1}M_0+p^k-p^{k+1}.
	\end{equation}
	We put $Q:=2a_1a_{p-1}^{\frac{p^k-1}{p-1}}\prod_{i=0}^{k}\frac{1}{\mu^{p^i}}$ and first  prove the important similarity relation 
\begin{equation}\label{key formula}	
	\frac{1}{\mu^{p^{k+1}}} a_{p-1}^{p^k}\sim \frac{Q^{p-1}}{\mu^{p^{k+1}-p^k}}.
\end{equation}

Note that \eqref{15a} is equivalent to $$\phi_k(0,p^k)\phi_k(p^k, p^{k+1})\sim - \prod_{j=0}^{p-1} \phi_k\left(jp^k,(j+1)p^k\right),$$
	and by Lemma~\ref{tri}(2) further to
	\begin{equation}\label{15d}
		\phi_k(p^k, p^{k+1})\sim -\prod_{j=1}^{p-1} \phi_k(jp^k,(j+1)p^k).
	\end{equation}

	By Lemma~\ref{new lemma}, for every $p^k$-divisible $r\geq 0$,  	\begin{equation*}
		\phi_k(r, r+p^k)\sim \frac{2a_1a_{p-1}^{\frac{p^k-1}{p-1}}}{1-\lambda^{r+p^k}}\prod_{i=0}^{k-1}\frac{1}{\mu^{p^i}}.
	\end{equation*}	

	By  Lemmas~\ref{tri}(2) and~\ref{mu}, this implies
	\begin{equation}\label{15e}
		\phi_k(r, r+p^k)\sim
		\begin{cases}
			\frac{Q}{r/p^k+1}&   	\textrm{if~}p^{k+1}\nmid \left(r+p^k\right),\  \\
			\frac{Q}{\mu^{p^{k+1}-p^k}}&	\textrm{for~}r=p^{k+1}-p^k,\\
			\frac{Q}{2\mu^{p^{k+1}-p^k}}&	\textrm{for~}r=2p^{k+1}-p^k.
		\end{cases}
	\end{equation}
	
	On the other hand, by \eqref{k}, for any $p^{k+1}\nmid r$ we have $$\v(\phi_k(r,r+(p-1)p^k))= (p-1)p^kM_0+p^k-p^{\vp(r+(p-1)p^{k})}.$$
	
By Lemma~\ref{key lemma}(2), this implies 
	$$\phi_k(r, r+(p-1)p^{k})\sim \frac{1}{\lambda(1-\lambda^{p^{r+(p-1)p^{k}}})}\binom{r+1}{\ualpha(r,r+(p-1)p^{k})}_p\ua^{\ualpha(r,r+(p-1)p^k)}.$$
	
	Lemmas~\ref{alpha} and \ref{mu} simplify this similarity relation to 
	$$\phi_k(r, r+(p-1)p^k)\sim \frac{1}{\lambda(1-\lambda^{ r+(p-1)p^k})}\times \frac{r}{p^k}\times\lambda a_{p-1}^{p^k},$$
	and in particular,
	\begin{equation}\label{15f}
		\phi_k(jp^k, (j+(p-1))p^k)\sim\begin{cases}
			\frac{1}{\mu^{p^{k+1}}}a_{p-1}^{p^k} &\textrm{for~} j=1,\\
			\frac{1}{2\mu^{p^{k+1}}}a_{p-1}^{p^k} &\textrm{for~} j=p+1,\\
			\frac{1}{(j-1)\mu^{p^k}}a_{p-1}^{p^k} &\textrm{for~} 2\leq j\leq p-1.
		\end{cases}
	\end{equation}
	Combined with \eqref{15d}, \eqref{15e} and the transitivity property of ``$\sim$'', this similarity relation implies
	$$\frac{1}{\mu^{p^{k+1}}} a_{p-1}^{p^k}\sim\phi_k(p^k, p^{k+1})\sim -\prod_{j=1}^{p-1} \phi_k(jp^k,(j+1)p^k)\sim -\frac{Q}{\mu^{p^{k+1}-p^k}} \prod_{i=1}^{p-2}\frac{Q}{i+1}.$$
	
	Considering $(p-1)!\equiv-1\pmod p$, we prove \eqref{key formula}.
	
	Now we are ready to study
	\begin{equation}
		\phi_{k+1}(0,2p^{k+1})=\sum_{\ubeta\in S_1(0, 2p)}\prod_{j=0}^{m(\ubeta)}\phi_k(\beta_jp^k,\beta_{j+1}p^k).
	\end{equation}

	Similar to \eqref{15b}, for every $\ubeta\in S_1(0,2p)$ we have
	\begin{equation}\label{19e}
		\v\left(  \prod_{j=0}^{m(\ubeta)}\phi_k(\beta_jp^k,\beta_{j+1}p^k) \right)
		\geq 2p^{k+1}M_0+p^k-p^{k+1},
	\end{equation}
	with equality if and only if for every $0\leq j\leq m(\ubeta)$, 
	$$\beta_{j+1}-\beta_j=1; \textrm{~or~} \beta_{j+1}-\beta_j=p-1\textrm{~and~}p\nmid \beta_{j}.$$
	
	Putting the restriction $p\nmid \beta_j$ for $1\leq j\leq m(\ubeta)$ in mind, the sequences $\ubeta$ that satisfy this equality in \eqref{19e} are
	\begin{equation}\label{17e}
		(0,1,2,p+1,2p) \textrm{~and~} (0,1,\dots,j, j+p-1,j+p,\dots, 2p)
	\end{equation}
	for all $2\leq j\leq p-1.$
	
	We first show that there is a cancellation between the terms in \eqref{19e} corresponding to $(0,1,2,p+1,2p)$ and $(0,1,2, p+1,p+2,\dots, 2p)$, i.e.
	\begin{multline}\label{17b}
		\v\Bigg(\phi_k(0,p^k)\phi_k(p^k,2p^k)\phi_k\left(2p^k,(p+1)p^k\right)\prod_{i=p+1}^{2p-1} \phi_k\left(ip^k,(i+1)p^k\right)	\\
		+\phi_k(0,p^k)\phi_k(p^k,2p^k)\phi_k\left(2p^k,(p+1)p^k\right)\phi_k\left((p+1)p^k, 2p^{k+1}\right)\Bigg)>2p^{k+1}M_0+p^k-p^{k+1}.
	\end{multline}
	
	Clearly, in order to prove \eqref{17b}, it is enough to show 
	\begin{multline*}
		\phi_k(0,p^k)\phi_k(p^k,2p^k)\phi_k\left(2p^k,(p+1)p^k\right)\prod_{i=p+1}^{2p-1} \phi_k\left(ip^k,(i+1)p^k\right)	\\
		\sim -\phi_k(0,p^k)\phi_k(p^k,2p^k)\phi_k\left(2p^k,(p+1)p^k\right)\phi_k\left((p+1)p^k, 2p^{k+1}\right).
	\end{multline*}
	
	By Lemma~\ref{tri}(2), it can be further reduced to show
	\begin{equation}\label{15g}
		\prod_{i=p+1}^{2p-1} \phi_k\left(ip^k,(i+1)p^k\right)
		\sim -\phi_k\left((p+1)p^k, 2p^{k+1}\right).
	\end{equation}
	
	By \eqref{15e}, we have
	$$ \prod_{i=p+1}^{2p-1} \phi_k\left(ip^k,(i+1)p^k\right)\sim \frac{Q}{2\mu^{p^{k+1}-p^k}} \prod_{i=p+1}^{2p-2}\frac{Q}{i+1}=-\frac{Q^{p-1}}{2\mu^{p^{k+1}-p^k}},$$
	where the last equality is from $(p-1)!\equiv-1\pmod p$.
	
	Combined with \eqref{key formula} and \eqref{15f}, this proves \eqref{15g} by the following chain of similarity relations
	$$ \prod_{i=p+1}^{2p-1} \phi_k\left(ip^k,(i+1)p^k\right)\sim -\frac{Q^{p-1}}{2\mu^{p^{k+1}-p^k}}\sim	-\frac{1}{2\mu^{p^{k+1}}}a_{p-1}^{p^k} \sim -\phi_k\left((p+1)p^k, 2p^{k+1}\right).$$
	
	Therefore, to prove  
	\begin{equation}\label{17d}
		\phi_{k+1}(0,2p^{k+1})=2p^{k+1}M_0+p^k-p^{k+1},
	\end{equation}
	it is enough to show that the summation over the rest $\ubeta$'s listed in \eqref{17e} satisfies 
	\begin{multline}\label{17f}
		\v\left(\sum_{j=3}^{p-1}\left(\phi_k\left(jp^k,jp^k+(p-1)p^k\right)\prod_{i=0}^{j-1} \phi_k\left(ip^k,(i+1)p^k\right)\prod_{\ell=j+p-1}^{2p-1} \phi_k\left(\ell p^k,(\ell+1)p^k\right)\right)\right)
		\\=2p^{k+1}M_0+p^k-p^{k+1}.
	\end{multline}

	
	Combining \eqref{15e} and \eqref{15f}, for every $3\leq j\leq p-1$ we have
	\begin{multline}\label{17a}
		\phi_k\left(jp^k,jp^k+(p-1)p^k\right)\prod_{i=0}^{j-1} \phi_k\left(ip^k,(i+1)p^k\right)\prod_{\ell=j+p-1}^{2p-1} \phi_k\left(\ell p^k,(\ell+1)p^k\right)
		\\
		\sim 	\frac{1}{(j-1)\mu^{p^k}}a_{p-1}^{p^k} \prod_{i=0}^{j-1}  \frac{Q}{i+1} \prod_{\ell=j+p-1}^{2p-2} \frac{Q}{\ell+1} \times \frac{Q}{2\mu^{p^{k+1}-p^k}}\\
			=\frac{a_{p-1}^{p^k}Q^{p+1}}{2(j-1)j\mu^{p^{k+1}}}\times \frac{1}{(p-1)!}=-\frac{a_{p-1}^{p^k}Q^{p+1}}{2(j-1)j\mu^{p^{k+1}}}.
	\end{multline}

	By Lemma~\ref{tri}(1), this gives
	$$\v\left(\frac{a_{p-1}^{p^k}Q^{p+1}}{\mu^{p^{k+1}}}\right)=2p^{k+1}M_0+p^k-p^{k+1}.$$
	Note that this equality can also   be obtained from the definition of $Q$.

	From our assumption $p\geq 5$, we have $$\sum_{j=3}^{p-1}\frac{1}{(j-1)j}=\sum_{j=3}^{p-1}\left(\frac{1}{j-1}-\frac{1}{j}\right)=\frac{1}{2}-\frac{1}{p-1}=\frac{3}{2}\neq 0,$$
and hence
	\begin{equation}\label{19t}
		\v\left(\sum_{j=3}^{p-1} -\frac{a_{p-1}^{p^k}Q^{p+1}}{2(j-1)j\mu^{p^{k+1}}}\right)=\v\left(\frac{a_{p-1}^{p^k}Q^{p+1}}{\mu^{p^{k+1}}}\right)=2p^{k+1}M_0+p^k-p^{k+1}.
	\end{equation}
	
	Combined with \eqref{17a}, this implies 
	$$	\sum_{j=3}^{p-1}\left(\phi_k(jp^k,(j+(p-1))p^k)\prod_{i=0}^{j-1} \phi_k(ip^k,(i+1)p^k)\prod_{\ell=j+p-1}^{2p-1} \phi_k(\ell p^k,(\ell+1)p^k)\right)
	\sim \sum_{j=3}^{p-1} -\frac{a_{p-1}^{p^k}Q^{p+1}}{2(j-1)j\mu^{p^{k+1}}}.$$
	By Lemma~\ref{tri}(1) and \eqref{19t}, the above similarity relation implies \eqref{17f}, and hence \eqref{17d}.
	
	Therefore, we have 
	$$M_{k+1}(0,2p^{k+1})=\frac{2p^{k+1}M_0+p^k-p^{k+1}}{2p^{k+1}}=M_0-\frac{p-1}{2p}\leq M_0-\frac{1}{p}.$$
	
	Since $M_\ell=M_0$ for every $0\leq \ell\leq k$, we prove that $f$ is $(k+1)$-dominant.
\end{proof}

\end{document}